\documentclass[12pt,a4paper]{article}

\usepackage{geometry}
\usepackage{amsmath}
\usepackage{amsthm}
\usepackage{amsfonts}
\usepackage{amssymb}
\usepackage{mathtools}
\usepackage[scr=rsfs]{mathalpha}
\usepackage[english]{babel}
\usepackage[utf8]{inputenc} 
\usepackage{lmodern} 
\usepackage{graphicx} 
\usepackage{xcolor}
\usepackage{float}
\usepackage{lastpage}
\usepackage{caption}
\usepackage{tensor}
\usepackage{braket}
\usepackage{tikz}
\usepackage{forest,textcomp}
\usepackage{cancel}
\usepackage[noadjust]{cite}
\usepackage{subcaption}
\usepackage{fancyref}
\usepackage{subfiles}
\usepackage{gensymb}
\usepackage{enumerate}
\usepackage{url}

\usepackage[colorlinks]{hyperref}
\hypersetup{linkcolor=blue, urlcolor=blue, citecolor=red}

\newcommand{\inv}[1]{{#1}^{-1}}
\newcommand{\dom}[1]{\operatorname{dom}(#1)}
\newcommand{\im}[1]{\operatorname{im}(#1)}
\newcommand{\Sym}[1]{\operatorname{Sym}(#1)}
\newcommand{\Stab}[1]{\operatorname{Stab}(#1)}
\newcommand{\AStab}[1]{{\operatorname{AStab}}(#1)}

\newcommand{\setStab}[2]{{#1}_{\lbrace #2\rbrace}}
\newcommand{\pointStab}[2]{{#1}_{(#2)}}

\newcommand{\supp}[1]{\operatorname{supp}(#1)}

\newcommand*{\zfc}{\ifmmode\mathbf{ZFC}\else\textbf{ZFC}\fi}
\newcommand{\fin}{\mathfrak{F}}
\newcommand{\partition}{\mathcal{P}}
\newcommand{\filter}{\mathcal{F}}

\newcommand{\card}[1]{\lvert{#1}\rvert}

\newcommand{\eval}[1]{\ensuremath{#1 |}}

\newcommand{\mutt}[2]{\mathcal{C}(#1,  #2)}
\newcommand{\mte}{\mathcal{T}}
\newcommand{\trans}[1]{#1^{T}}

\renewcommand{\set}[2]{\{#1 : #2\}}
\newcommand*\from{\colon}

\newcommand{\authorname}[1]{\textsc{#1}}
\newcommand{\affiliation}[1]{\small\textsc{#1}}
\newcommand{\email}[1]{\small\textit{E-mail address}: \href{mailto:#1}{\nolinkurl{#1}}}

\theoremstyle{plain}
\newtheorem{theorem}{Theorem}[section]
\newtheorem{lemma}[theorem]{Lemma}
\newtheorem{proposition}[theorem]{Proposition}

\newtheorem{corollary}[theorem]{Corollary}

\theoremstyle{definition}

\numberwithin{equation}{section}

\newenvironment{nalign}{
	\begin{equation}
		\begin{aligned}
		}{
		\end{aligned}
	\end{equation}
	\ignorespacesafterend
}

\title{Maximal Subsemigroups of Infinite Symmetric Inverse Monoids}

\author{M. Hampenberg and Y. P\'eresse}

\date{}

\begin{document}

\maketitle

\abstract{The symmetric inverse monoid $I_X$ on a set $X$ consists of all bijective functions whose domain and range are subsets of $X$ under the usual composition and inversion of partial functions. For an arbitrary infinite set $X$, we classify all maximal subsemigroups and maximal inverse subsemigroups of $I_X$ which contain the symmetric group $\Sym{X}$ or any of the following subgroups of $\Sym{X}$: the pointwise stabiliser of a finite subset of $X$, the stabiliser of an ultrafilter on $X$, or the stabiliser of a partition of $X$ into finitely many parts of equal cardinality.}

\renewcommand{\thefootnote}{\fnsymbol{footnote}} 
\renewcommand{\thefootnote}{\arabic{footnote}} 

\section{Introduction}    
A \emph{maximal subalgebra} of an algebra  $A$ (in the sense of Universal Algebra) is a maximal element of the set of all proper subalgebras of $A$ ordered by containment.  In other words, a proper subalgebra $M$ of $A$ is maximal if the subalgebra generated by $M$ and any element of $A \setminus M$ equals $A$. Given an algebra $A$ of interest, it is a natural problem to try to classify its maximal subalgebras and doing so would be an important step towards understanding the subalgebra structure of $A$ in general. 

In Group Theory, perhaps the most well-known such classification is the O'Nan-Scott Theorem, see \cite{scott1980representations, aschbacher1985maximal, liebeck1988nan}, which classifies the maximal subgroups of the symmetric group $\Sym{X}$ on a finite set $X$. 
Among these maximal subgroups are the setwise stabilisers
\begin{equation*}
    \setStab{\Sym{X}}{\Sigma}= \set{f \in \Sym{X} }{ \Sigma f = \Sigma}
\end{equation*}
of proper, non-empty subsets $\Sigma$ of $X$ with $|\Sigma|\neq |X\setminus \Sigma|$; and the  stabilisers $\Stab{\partition}$ of partitions $\partition=\{\Sigma_0, \dots, \Sigma_{n-1}\}$ of $X$ into $n=\{0, \dots, n-1\}$ blocks of equal cardinality, defined by
    \begin{equation*}
        \Stab{\partition} = \set{f \in \Sym{X} }{ (\forall i \in n) (\exists j \in n) (\Sigma_i f = \Sigma_j)}.
    \end{equation*}
Such stabilisers of sets and partitions are the only \emph{imprimitive} maximal subgroups of $\Sym{X}$, that is, the only maximal subgroups which preserve a partition of $X$. The O'Nan-Scott Theorem also classifies the various types of primitive maximal subgroups of $\Sym{X}$ when $X$ is finite, but we omit their descriptions here as they are less relevant for the following discussion. 

Much work has also been done on maximal subgroups of infinite symmetric groups. It seems unlikely that all maximal subgroups of $\Sym{X}$ for infinite $X$ can be classified in a meaningful way. However, several concrete classes of maximal subgroups of $\Sym{X}$ have been found and there are classifications of certain types of maximal subgroups. In particular, the imprimitive maximal subgroups described above have more or less direct analogues in the case of infinite $X$.
The setwise stabilser $\setStab{\Sym{X}}{\Sigma}$ is a maximal subgroup of $\Sym{X}$ for every non-empty finite subset $\Sigma$ of an infinite set $X$, see \cite{ball1966maximal}. Setwise stabilisers of distinct sets are distinct and so the $|X|$ many finite subsets of $X$ give rise to $|X|$ many maximal subgroups of $\Sym{X}$. 
The analogue of stabilisers of partitions is less direct. Recall that a \emph{moiety} of an infinite set $X$ is a subset $Y$ of $X$ such that $|Y|=|X\setminus Y|=|X|$. A \emph{finite partition into moieties} of $X$ is a partition $\partition=\{\Sigma_0, \dots, \Sigma_{n-1}\}$ of $X$ into finitely many moieties. As discussed in a note added in proof to \cite{subgroups_macpherson_neumann}, the stabiliser $\Stab{\partition}$ of a finite partition into moieties of $X$ is not a maximal subgroup of $\Sym{X}$ since it is strictly contained in the \emph{almost stabiliser} $\AStab{\partition}$ of $\partition$, defined as
    \begin{align*}
        \AStab{\partition} = \set{f \in \Sym{X}}{\;&(\forall i \in n) (\exists j \in n)\\
        &(|\Sigma_i f \setminus \Sigma_j| +|\Sigma_j \setminus \Sigma_i f| < |X|)}.
    \end{align*}
However, $\AStab{\partition}$ is a maximal subgroup of $\Sym{X}$. An infinite set $X$ has $2^{|X|}$ many finite partitions into moieties. Note that distinct finite partitions into moieties always have distinct stabilisers but may have the same almost stabiliser. However, it can be shown that there are indeed $2^{|X|}$ many distinct almost stabilisers of finite partitions into moieties. 
There is a lager class still of maximal subgroups of $\Sym{X}$ for infinite $X$, which does not have an analogue for finite $X$: it was shown in \cite{ richman1967maximal} that, for any ultrafilter $\filter$ on $X$ (as defined in Section \ref{section:sets}), the stabiliser
$$\setStab{\Sym{X}}{\filter}=\set{f \in \Sym{X}}{(\forall \Sigma \subseteq X)(\Sigma \in \filter \iff \Sigma f \in \filter)}$$
is maximal in $\Sym{X}$. Distinct ultrafilters give rise to distinct stabiliser subgroups and, by Pospi\u sil's Theorem (see, for example,  \cite[Theorem 7.6]{Jech2003}), there are $2^{2^{|X|}}$ many ultrafilters on an infinite set $X$. Since there are also $2^{2^{|X|}}$ subsets of $\Sym{X}$, it follows that $\Sym{X}$ has precisely $2^{2^{|X|}}$ many maximal subgroups when $X$ is infinite. Note that the setwise stabiliser $\setStab{\Sym{X}}{\Sigma}$ of a subset $\Sigma$ of $X$ can also be seen as the stabiliser of the filter $\filter$ of all subsets of $X$ which contain $\Sigma$, but $\filter$ is only an ultrafilter when $|\Sigma|=1$. It was shown in \cite{brazil1994maximal} that every maximal subgroup of $\Sym{X}$ which contains the pointwise stabiliser 
$$\pointStab{\Sym{X}}{\Sigma}=\set{f \in \Sym{X} }{ (\forall x \in \Sigma) (xf=x)}$$ of a set $\Sigma\subseteq X$ with $|X \setminus \Sigma|=|X|$ is either the stabiliser of a finite partition into moieties of $X$ or the stabiliser of a certain kind of filter on $X$.
For more results on maximal subgroups of $\Sym{X}$ when $X$ is infinite, see for example \cite{ball1968indices, baumgartner1993maximal, biryukov2000set, covington1996some, macpherson1993large, maximal_macpherson_preager}.

Symmetric groups take a central role within Group Theory for at least two good reasons. First, the symmetries of any object $X$ are naturally represented by a subgroup of $\Sym{X}$, linking Group Theory to virtually any other area of mathematics. Second, by Cayley's Theorem, every group is isomorphic to a subgroup of $\Sym{X}$ for some $X$. In this context, 
the symmetric group has two natural and widely studied analogues in the classes of semigroups and inverse semigroups. Recall that a \emph{semigroup} is a set together with an associative binary operation, a \emph{monoid} is a semigroup which has an identity element, and an \emph{inverse semigroup} is a semigroup $S$ such that for every $s\in S$ there exists a unique $s^{-1}\in S$ (called the inverse of $s$) satisfying $ss^{-1}s=s$ and $s^{-1}ss^{-1}=s^{-1}$. The \emph{full transformation semigroup} $X^X$ consists of all functions from a set $X$ to itself; and the \emph{symmetric inverse monoid} $I_X$ consists of all \emph{partial permutations} on $X$, that is bijections whose domain and co-domain are subsets of $X$. In both cases the semigroup multiplication is the usual composition of partial functions (see Section \ref{section:partial_functions} for details). Functions and partial permutations are two different ways of generalising the classical concept of symmetries, allowing for some degree of information loss. Moreover, analogously to Cayley's Theorem, every semigroup is isomorphic to a subsemigroup of $X^X$ and every inverse semigroup is isomorphic to an inverse subsemigroup of $I_X$ for some $X$, see, for example, \cite[Theorem 1]{lawson_inverse_semigroups}.

In the case that $X$ is finite,  it is not difficult to classify the maximal subsemigroups of $X^X$ and $I_X$ in terms of the maximal subgroups of $\Sym{X}$.  The argument is as follows. If $S$ is either $X^X$ or $I_X$,  then $S \setminus \Sym{X}$ is an ideal of $S$ and so the maximal subsemigroups of $S$ which do not contain $\Sym{X}$ are precisely subsemigroups of the form $(S \setminus \Sym{X}) \cup G$ where $G$ is a maximal subgroup of $\Sym{X}$.  Moreover,  if $f\in S$ has image size $|X|-1$, then the semigroup generated by $\Sym{X}$ and $f$ is $S$.  Hence the only maximal subsemigroup of $S$ which does contain $\Sym{X}$ is the union of $\Sym{X}$ and the set of elements of $S$ with image size at most $|X|-2$.  Note that in the case when $S=I_X$, all maximal subsemigroups of $I_X$ are also inverse subsemigroups and so the maximal subsemigroups and maximal inverse subsemigroups of $I_X$ coincide when $X$ is finite.  A classification of the maximal subsemigroups of $I_X$ and maximal subsemigroups of certain subsemigroups of $I_X$ when $X$ is finite was also given in \cite{xiuliang1999classification}.  Graham, Graham, and Rhodes showed that the maximal subsemigroups of every finite semigroup $S$ are, in some sense,  determined by the maximal subgroups of $S$, see \cite{graham1968maximal}.

We turn our attention to maximal subsemigroups of $X^X$ and $I_X$ for infinite $X$. We just saw that, in the case of finite $X$, every maximal subsemigroup $M$ of $X^X$ or $I_X$ corresponds to a large subgroup of $\Sym{X}$. More precisely, $M \cap \Sym{X}$ is either a maximal subgroup of $\Sym{X}$ or $\Sym{X}$ itself and, conversely,  each such subgroup corresponds to exactly one maximal subsemigroup. This motivates the following approach for infinite $X$: Consider a ``large'' (for example maximal) subgroup $G$ of $\Sym{X}$ and aim to classify all maximal subsemigroups of $X^X$ or $I_X$ which contain $G$. The maximal subsemigroups of $X^X$ containing $\Sym{X}$ where classified in \cite{gavrilov1965functional, pinsker2005maximal} and those containing the pointwise stabiliser of a finite set, the stabiliser of an ultrafilter, or the stabiliser of a finite partition into moieties where classified in  \cite{maximal_east_mitchell_peresse}.

In this paper we let $X$ be an infinite set and classify the maximal subsemigroups and maximal inverse subsemigroups of $I_X$ which contain $\Sym{X}$ (Theorem \ref{thm:maximal_sym}), the pointwise stabiliser of a finite subset of $X$ (Theorem \ref{thm:maximal_pointwise}),  the stabiliser of an ultrafilter on $X$ (Theorem \ref{thm:maximal_ultrafilter}),  or the stabiliser of a finite partition of $X$ into moieties (Theorem \ref{thm:maximal_astab}).
Along they way, we prove the following two corollaries to our main results. We determine the intersection of all maximal subsemigroups of $I_X$, which turns out to coincide with the intersection of all its maximal inverse subsemigroups (Corollary \ref{cor:intersection}). We also show that, analogously to symmetric groups, $I_X$ has $2^{2^{|X|}}$ many maximal subsemigroups and maximal inverse subsemigroups (Corollary \ref{cor:count}).
In order to provide clear and concise proofs for our main results, we will establish a general method (Theorem \ref{lem:total->all}) for classifying  maximal subsemigroups and maximal inverse subsemigroups of $I_X$ containing any ``sufficiently large'' inverse subsemigroup of $I_X$ and set up a way of leveraging certain results about $X^X$ from \cite{maximal_east_mitchell_peresse} to prove analogous results about $I_X$ (Section \ref{fromIXtoXX}). Alternative direct proofs (not using results about $X^X$) of our results in the case when $X$ is countable may be found in the first author's PhD thesis \cite[Chapter 3]{hampenberg2025submonoids}.

\section{Preliminaries}

In this section we introduce some terms and notation and prove a number of results that will be used in the proofs of the main results. 

\subsection{Sets and filters}\label{section:sets}
We assume the axiom of choice throughout this paper. The cardinality of a set $X$ is denoted by $|X|$. The successor cardinal of a cardinal $\lambda$ is the least cardinal strictly greater than $\lambda$ and is written $\lambda^+$.

A \emph{filter} on a set $X$ is a non-empty family $\filter$ of subsets of $X$ such that
    \begin{enumerate}[\normalfont(i)]
        \item $\emptyset \notin \filter$;
        \item If $\Sigma \in \filter$ and $\Sigma \subseteq \Gamma \subseteq X$, then $\Gamma \in \filter$;
        \item If $\Sigma, \Gamma \in \filter$, then $\Sigma \cap \Gamma \in \filter$.
    \end{enumerate}
A filter $\filter$ on $X$ is an \emph{ultrafilter} if $\filter$ is maximal with respect to containment in the set of all filters on $X$. Equivalently, $\filter$ is an ultrafilter if and only if for every $\Sigma \subseteq X$ either $\Sigma \in \filter$ or $X \setminus \Sigma \in \filter$. A \emph{principal ultrafilter} is a filter of the form $\set{\Sigma \subseteq X}{x \in \Sigma}$ for some $x \in X$. All other ultrafilters on $X$ are \emph{non-principal}. An ultrafilter on $X$ is \emph{uniform} if every element of the filter has cardinality $|X|$.

\subsection{Partial functions, full transformations, and partial bijections}\label{section:partial_functions}

In this paper, we will prove results about the symmetric inverse monoid $I_X$ using results about the full transformation semigroup $X^X$.  Since $I_X$ and $X^X$  are both subsemigroups of the semigroup $P_X$ of all partial functions (defined just below),  it is convenient to state the following definitions and results for $P_X$.

If $X$ is a set, then the semigroup $P_X$ of \emph{partial functions} consists of all functions whose domain and range are subsets of $X$. In other words, an element $f$ of $P_X$ is a subset of the Cartesian product $X \times X$ such that for every $x \in X$ there exists at most one $y \in X$ such that $(x,y) \in f$. If $x,y \in X$, then we may also write $(x)f=y$ (or simply $xf=y$) to denote that $(x,y) \in f$.  The operation on the semigroup $P_X$ is composition of partial functions, a generalisation of the usual composition of functions defined as follows. If $f,g \in P_X$, then 
$$fg=\set{(x,y) }{ (\exists z \in X)((x,z) \in f \text { and } (z,y) \in g)}.$$
Note that we compose from left to right in this paper, which is why we write elements of $P_X$ to the right of their arguments. 

Let $f \in P_X$. The \emph{domain} $\dom{f}$ and \emph{image} $\im{f}$ of $f$ are
\begin{align*}
\dom{f}&=\set{x \in X }{ (x,y) \in f \text{ for some }y \in X}\\
\im{f}&=\set{ y \in X }{ (x,y) \in f \text{ for some } x \in X}.
\end{align*}
We say that $f$ is \emph{total} if $\dom{f}=X$ and $f$ is \emph{surjective} if $\im{f}=X$. As usual, $f$ is \emph{injective} if $xf=yf \implies x=y$ for all $x,y \in\dom{f}$.
If $A$ is a subset of $X$, then 
$$Af=\set{(x)f}{x \in A \cap \dom{f}}$$
and the \emph{restriction} of $f$ to $A$ is 
$$f|_{A}=\set{(x,y) \in f}{x \in A}.$$
We may now define
\begin{align*}
X^X&=\set{f \in P_X }{ \dom{f}=X};\\
I_{X}&=\set{f \in P_X }{ f \text{ is injective}};\\
\Sym{X} &=\set{f \in I_X }{ \dom{f}=\im{f}=X}.
\end{align*}
The \emph{support} of a permutation $a\in \Sym{X}$ is $\supp{a}=\set{x\in X}{(x)a\neq x}$.

\subsection{Rank, defect, and collapse}
The \emph{rank} of $f\in P_X$ is $r(f)=|\im{f}|$ and the \emph{defect} of $f$ is $d(f)=|X \setminus \im{f}|$. 
A \emph{transversal} of $f$ is a transversal in the usual combinatorial sense of the pre-images of $f$. That is,  $T \subseteq X$ is a transversal of $f$ if for every $y\in \im{f}$ there exists a unique $x\in T$ such that $xf=y$.
The \emph{collapse} of $f$ is $c(f)=|X \setminus T|$ where $T$ is any transversal of $f$.  

Note that $d(f)=0$ if and only if $f$ is surjective and $c(f)=0$ if and only if $f$ is total and injective for any $f \in P_X$.  If $f\in I_X$,  then $r(f)=|\im{f}|=|\dom{f}|$ and $c(f)=d(f^{-1})$.
The following lemma establishes some basic properties of how rank, collapse, and defect behave under composition of elements of $I_X$ and will be used frequently throughout the rest of the paper.

\begin{lemma} \label{lem:defect_properties}
    Let $X$ be a set, $\mu$ an infinite cardinal such that $\mu \leq \card{X}$, and $f,g \in I_X$. 
    Then the following hold:
    \begin{enumerate}[~\normalfont(i)]
            \item $r(fg) \leq \min(r(f),r(g))$;
        \label{lem:defect_properties/rank_composition}
    
        \item $c(f) \leq c(fg) \leq c(f)+c(g)$;
        \label{lem:defect_properties/collapse_composition}

        \item $d(g) \leq d(fg) \leq d(f)+d(g)$;
        \label{lem:defect_properties/defect_composition}
        
        \item if $f$ is \emph{surjective} (i.e. $d(f)=0$), then $c(fg) = c(f)+c(g)$;
        \label{lem:defect_properties/collapse_surjective}
        
        \item if $g$ is \emph{total} (i.e. $c(g)=0$), then $d(fg) = d(f)+d(g)$;
        \label{lem:defect_properties/defect_total}

        \item if $d(f) < \mu \leq c(g)$, then $c(fg) \geq \mu$;
        \label{lem:defect_properties/defect<collapse}

        \item if $c(g) < \mu \leq d(f)$, then $d(fg) \geq \mu$.
        \label{lem:defect_properties/collapse<defect}

    \end{enumerate}
\end{lemma}

\begin{proof}
    \begin{enumerate}[\normalfont(i)]
    
            \item This follows from the fact that $\dom{fg} \subseteq \dom{f}$ and $\im{fg} \subseteq \im{g}$.
    
        \item Since $\dom{fg} \subseteq \dom{f}$, it follows that $c(f) \leq c(fg)$. For the second inequality, note that $x \in X \setminus \dom{fg}$ if and only if either $x\in X \setminus \dom{f}$ or $x\in \dom{f}$ and $(x)f \in X \setminus \dom{g}$. Thus $X \setminus \dom{fg}$ is the disjoint union of $X \setminus \dom{f}$ and $(X \setminus \dom{g})f^{-1}$.
In particular, 
        \begin{nalign}\label{equation:collapse}
c(fg) &= \card{X \setminus \dom{fg}}\\ 
&= \left|X \setminus \dom{f}\right| + 
        \left|(X \setminus \dom{g})f^{-1}\right|\\
        &\leq c(f) + c(g).
        \end{nalign}

        \item This follows from \eqref{lem:defect_properties/collapse_composition} since $d(h)=c(h^{-1})$ for every $h\in I_X$.

        \item If $f$ is surjective, then $f^{-1}$ is total and so the inequality equation in equation \eqref{equation:collapse} becomes an equality.

        \item This follows from \eqref{lem:defect_properties/collapse_surjective} and the identity $d(h)=c(h^{-1})$.

        \item If $d(f)<\mu\leq c(g)$, then $|X \setminus \dom{g}|=c(g)\geq \mu$ and $|X \setminus \dom{f^{-1}}|=c(f^{-1})=d(f)<\mu$. 
        Thus $|(X \setminus \dom{g})\cap \dom{f^{-1}}|\geq \mu$. If follows that $|(X \setminus \dom{g})f^{-1}|\geq \mu$ and so the result follows from equation \eqref{equation:collapse}.
        
\item This follows from \eqref{lem:defect_properties/defect<collapse}. 

    \end{enumerate}
\end{proof}

Note that by Lemma \ref{lem:defect_properties} part \eqref{lem:defect_properties/rank_composition} the set 
$$\fin=\set{f \in I_X }{ r(f)<|X|}$$
is an ideal of $I_X$ and so $S\cup \fin$ is a subsemigroup of $I_X$ whenever $S$ is a subsemigroup of $I_X$.

\subsection{Containment in maximal subalgebras}
Throughout this paper, the term ``algebra'' is used in the sense of Universal Algebra. The results in this subsection will be required for classifications of maximal (inverse) subsemigroups in later sections but also apply in the more general setting of maximal subalgebras of algebras. 

It is shown in \cite{baumgartner1993maximal} that when $X$ is infinite, and under certain set theoretic assumptions, there exists a subgroup of $\Sym{X}$, which is not contained in any maximal subgroup of $\Sym{X}$.
However,  constructing such subgroups is difficult and there are several sufficient conditions which imply that a subgroup of $\Sym{X}$ must be contained in a maximal subgroup; see \cite{baumgartner1993maximal, subgroups_macpherson_neumann, maximal_macpherson_preager}.  Analogous conditions implying that a given subsemigroup of $X^X$ is contained in a maximal subsemigroup were given in \cite{maximal_east_mitchell_peresse}.  
In this paper, we require one such sufficient condition involving the concept of relative rank, which we now define.  

If $S$ is an algebra and $G$ a subalgebra of $S$, then  the \emph{relative rank} of $G$ in $S$ is the least cardinality of a set $U\subseteq S$ such that $G \cup U$ is a generating set for $S$.  It was shown in \cite[Lemma 6.9]{ subgroups_macpherson_neumann} that every subgroup $G$ of a group $S$ which has finite relative rank in $S$ is contained in a maximal subgroup of $S$.  (In fact,  it is sufficient for $G$ to have relative rank strictly less than the cofinality of $S$.) This result easily generalises to all algebras:

\begin{proposition} \label{prop:contained_relative_rank}
    Let $S$ be an algebra and $G$ a proper subalgebra of $S$.
    If $G$ has finite relative rank in $S$, then $G$ is contained in a maximal subalgebra of $S$.
\end{proposition}

\begin{proof}
    This proof is an application of Zorn's lemma.
    Let $\mathcal{A}$ be the partially ordered set of all proper subalgebras of $S$ containing $G$, ordered by inclusion.
    Given a chain $(A_i)_{i \in I}$ in $\mathcal{A}$, the union $\bigcup_{i \in I} A_i$ is a subalgebra of $S$. If $\bigcup_{i \in I} A_i$ is a proper subalgebra of $S$, then $(A_i)_{i \in I}$ is bounded by an element in $\mathcal{A}$, namely $\bigcup_{i \in I} A_i$.
So suppose that $\bigcup_{i \in I} A_i = S$. Since $G$ has finite relative rank in $S$, there exists a finite $F \subseteq S$ such that $G\cup F$ generates $S$.
    Since $(A_i)_{i \in I}$ is ordered by inclusion and $\bigcup_{i \in I} A_i = S$, there exists $j \in I$ such that $F \subseteq A_j$.
    But $G \subseteq A_j$ and so $A_j = S$, contradicting the assumption that $(A_i)_{i \in I}$ is a chain in $\mathcal{A}$.
    Hence every chain in $\mathcal{A}$ is bounded from above, which by Zorn's lemma implies that $\mathcal{A}$ contains a maximal element.
\end{proof}

Given a subalgebra $G$ of an algebra $S$, it is natural to ask for a classification of all maximal subalgebras of $S$ which contain $G$. The next lemma may be used for this purpose. 

\begin{lemma} \label{lem:exactly_maximal}
    Let $S$ be an algebra, $G$ a subalgebra of $S$, and $\mathcal{M}$ be a set of proper subalgebras of $S$ such that $G \subseteq M$ for all $M \in \mathcal{M}$.
    If every $U \subseteq S$ which is not contained in any element of $\mathcal{M}$ satisfies that $G \cup U$ is a generating set for $S$, then the maximal subalgebras of $S$ containing $G$ are exactly the maximal elements of $\mathcal{M}$ under containment.
\end{lemma}

\begin{proof} 
Let $S$, $G$, and $\mathcal{M}$ satisfy the conditions of the lemma. We first show that all maximal subalgebras of $S$ containing $G$ must be elements of $\mathcal{M}$. Aiming for a contradiction, suppose that $N$ is a maximal subalgebra of $S$ such that $G \subseteq N$ and $N \notin \mathcal{M}$.
Since $N$ is maximal in $S$ and $\mathcal{M}$ consists of proper subalgebras of $S$, it follows that $N$ is not contained in any element of $\mathcal{M}$. Thus $G \cup N = N$ generates $S$, which gives the required contradiction.

Clearly non-maximal elements of $\mathcal{M}$ cannot be maximal subalgebras of $S$, so it only remains to show that every maximal element $M$ of $\mathcal{M}$ is a maximal subalgebra of $S$.
Let $s \in S \setminus M$ be arbitrary. Then $M \cup \{s\}$ is not contained in any of the elements of $\mathcal{M}$ (because $M$ is maximal in $\mathcal{M}$) and $G \subseteq M$ by assumption. Hence $M \cup \{s\}$ generates $S$ and so $M$ is maximal in $S$.
\end{proof}

\subsection{Classifying maximal subsemigroups and inverse subsemigroups}

The main aim of this subsection is to prove Theorem \ref{lem:total->all}, which allows us to classify certain kinds of maximal subsemigroups and maximal inverse subsemigroups of $I_X$. Along the way, we  provide a method to obtain maximal inverse subsemigroups from maximal subsemigroups of a general inverse semigroup (Lemma \ref{lem:inverse_intersection}).

\begin{lemma} \label{lem:intersection_of_inverses}
    Let $S$ be an inverse semigroup and $M$ a subsemigroup of $S$.
    Then $M \cap \inv{M}$ is the largest (with respect to containment) inverse subsemigroup of $S$ which is contained in $M$.
\end{lemma}

\begin{proof}
Since $M$ is a semigroup, so are $M^{-1}$ and $M\cap M^{-1}$. Moreover, $M \cap \inv{M}$ is clearly inverse and contained in $M$. Let $V$ be an inverse subsemigroup of $S$ contained in $M$ and $v\in V$ be arbitrary. 
Since $V$ is inverse $v^{-1} \in V \subseteq M$ and so $v, v^{-1} \in M$. Hence $v \in M\cap M^{-1}$ and so $V \subseteq M\cap M^{-1}$ as required. 
\end{proof}

\begin{lemma} \label{lem:inverse_intersection}
Let $S$ be an inverse semigroup, $G$ an inverse subsemigroup of $S$ of finite relative rank, and $\mathcal{M}$ the set of all maximal subsemigroups of $S$ containing $G$.
Then the maximal inverse subsemigroups of $S$ containing $G$ are precisely the maximal elements of $\set{M \cap M^{-1} }{ M \in \mathcal{M}}$ under containment.
\end{lemma}

\begin{proof}
First we show that every maximal element of $\mathcal{M}^*=\set{M \cap M^{-1} }{ M \in \mathcal{M}}$ is a maximal inverse subsemigroup of $S$ containing $G$.
By Lemma \ref{lem:intersection_of_inverses}, $\mathcal{M}^*$ consists of inverse subsemigroups of $S$. Since every $M \in \mathcal{M}$ contains $G=G^{-1}$, so does every $M\cap M^{-1} \in \mathcal{M}^*$.
Suppose that $N\cap N^{-1}$ is maximal in $\mathcal{M}^*$ and let $V$ be a proper inverse subsemigroup of $S$ containing $N\cap N^{-1}$. 
By assumption, $G$ is an inverse subsemigroup of $S$ of finite relative rank, and so there exists a finite $F \subseteq S$ such that $F \cup G$ generates $S$ as an inverse semigroup. Of course, $G$ is also an ordinary subsemigroup of the semigroup $S$ and so $(F \cup G) \cup (F \cup G)^{-1}=(F \cup F^{-1}) \cup G$ is a generating set for $S$ as a semigroup. In other words, $G$ has finite relative rank in $S$ as a semigroup.  
Thus, since $G \subseteq N \cap \inv{N} \subseteq V$, it follows that $V$ has finite relative rank (as a semigroup) in $S$. Hence Proposition \ref{prop:contained_relative_rank} implies that $V$ must be contained in a maximal subsemigroup $M$ of $S$. Since $G \subseteq V \subseteq M$ it follows that $M \in \mathcal{M}$. By Lemma \ref{lem:intersection_of_inverses}, the largest inverse semigroup contained in $M$ is $M\cap M^{-1}$ and so $N \cap \inv{N} \subseteq V \subseteq M \cap \inv{M}$. 
    Since we assumed $N\cap \inv{N}$ to be a maximal element of $\mathcal{M}^*$, it follows that $N \cap \inv{N} = M \cap \inv{M}$. In particular, $V=N \cap \inv{N}$. We have shown that $N \cap N^{-1}$ is a maximal inverse subsemigroup of $S$.

    It remains to show that every maximal inverse subsemigroup $U$ of $S$ containing $G$ lies in $ \mathcal{M}^*$. Since $U$ is maximal as an inverse subsemigroup, $U$ has finite relative rank as an ordinary subsemigroup of $S$.  Thus it follows from Proposition \ref{prop:contained_relative_rank} that $U$ must be contained in some maximal subsemigroup $M$ of $S$. Since $G\subseteq U \subseteq M$, it follows that $M \in \mathcal{M}$. By Lemma \ref{lem:intersection_of_inverses}, $U$  is contained in the inverse semigroup $M\cap \inv{M}$. Hence $U = M \cap \inv{M}\in \mathcal{M}^*$. 
Clearly $U$ is maximal in $\mathcal{M}^*$ since $\mathcal{M}$ consists of inverse subsemigroups of $S$.
\end{proof}

\begin{lemma} \label{lem:in-and-out}
    Let $X$ be an infinite set, $Y$ a subset of $X$, $G$ a subset of $I_X$ such that $\inv{G} = G$, and $\mathcal{M}$ a collection of subsets of $I_X$.
    Then the following are equivalent:
    \begin{enumerate}[~\normalfont(i)]
        \item For all subsemigroups $U$ of $I_X$ such that $G \subseteq U$ and $U \nsubseteq M$ for all $M \in \mathcal{M}$, there exists $f \in U$ such that $\dom{f} = X$ and $\im{f}\subseteq Y$.
        \item For all subsemigroups $U$ of $I_X$ such that $G \subseteq U$ and $U \nsubseteq \inv{M}$ for all $M \in \mathcal{M}$, there exists $g \in U$ such that $\dom{g}\subseteq Y$ and $\im{g} = X$.
    \end{enumerate}
\end{lemma}

\begin{proof}
    (i)$\implies$(ii): Assume (i) and let $U$ be a subsemigroup of $I_X$ such that $G \subseteq U$ and $U \nsubseteq \inv{M}$ for all $M \in \mathcal{M}$.
    Then $G =G^{-1}\subseteq \inv{U}$ and $\inv{U} \nsubseteq M$ for all $M \in \mathcal{M}$.
    So by (i) there exists $f \in \inv{U}$ such that $\dom{f} = X$ and $\im{f} \subseteq Y$.
    Thus $g=f^{-1}\in U$ satisfies the required conditions.

    (ii)$\implies$(i): This follows by a symmetric argument.
\end{proof}

Recall that a moiety of an infinite set $X$ is a subset $Y$ of $X$ such that $|Y|=|X\setminus Y|=|X|$.

\begin{theorem} \label{lem:total->all}
    Let $X$ be an infinite set, $Y$ a moiety of $X$, and $G$ an inverse subsemigroup of $I_X$ satisfying
\begin{equation}\label{lem:total->all/YSym}
\Sym{Y} \subseteq \set{g|_{Y} }{ g \in G}.
\end{equation}
A collection $\mathcal{M}$ of subsets of $I_X$ is the set of all maximal subsemigroups of $I_X$ which contain $G$ if and only if the following five conditions hold. 
    \begin{enumerate}[~\normalfont(i)]
        \item $G \subseteq M$ for all $M \in \mathcal{M}$.
        \label{lem:total->all/G_subset}
        \item $M$ is a proper subsemigroup of $I_X$ for all $M \in \mathcal{M}$.
        \label{lem:total->all/M_subsemigroup}
        \item $M \nsubseteq N$ for all distinct $M,N \in \mathcal{M}$.
        \label{lem:total->all/M_not_contained}
        \item $\inv{M} \in \mathcal{M}$ for all $M \in \mathcal{M}$.
        \label{lem:total->all/M_inverses}
        \item For all subsemigroups $U$ of $I_X$ such that $G \subseteq U$ and $U \nsubseteq M$ for all $M\in \mathcal{M}$, there exists a total $f \in U$ such that $\im{f}\subseteq Y$.
        \label{lem:total->all/in-out-condition}
    \end{enumerate}
Moreover, if the five conditions above are satisfied,  then the maximal inverse subsemigroups of $I_X$ which contain $G$ are precisely the maximal elements of $\set{M\cap M^{-1}}{M \in \mathcal{M}}$ under containment. 
\end{theorem}

\begin{proof}    

First we will  show that the five conditions together imply that $\mathcal{M}$ is the set of all maximal subsemigroups of $I_X$ which contain $G$. By conditions \eqref{lem:total->all/G_subset} and \eqref{lem:total->all/M_subsemigroup} all elements of $\mathcal{M}$ are proper subsemigroups of $I_X$ and contain $G$ and by part \eqref{lem:total->all/M_not_contained},  every element of $\mathcal{M}$ is maximal in $\mathcal{M}$.  Hence, by Lemma \ref{lem:exactly_maximal}, it suffices to let  $U$ be an arbitrary subsemigroup of $I_X$ such that $G \subseteq U$ and $U \nsubseteq M$ for all $M \in \mathcal{M}$ and show that $U = I_X$.
By condition \eqref{lem:total->all/in-out-condition} there exists $f \in U$ such that $\dom{f} = X$ and $\im{f} \subseteq Y$. Note that, by condition \eqref{lem:total->all/M_inverses}, we also have that $U \nsubseteq \inv{M}$ for all $M \in \mathcal{M}$. Hence Lemma \ref{lem:in-and-out} implies that there exists $g \in U$ such that $\dom{g}\subseteq Y$ and $\im{g} = X$. Note that $\im{f^2}$ is a moiety of $Y$ since $Y$ is a moiety of $X$ and $\im{f}\subseteq Y$. Similarly, $\dom{g^2}$ is a moiety of $Y$. So, by replacing $f$ and $g$ with $f^2$ and $g^2$ if necessary, we may assume that $\im{f}$ and $\dom{g}$ are moieties of $Y$.

We will now show that $f G g=I_X$. Let $h \in I_X$ be arbitrary and define $p=f^{-1}hg^{-1}$. Then $\dom{p}\subseteq \im{f}$ and $\im{p} \subseteq \dom{g}$. Thus $p$ is a partial permutation of $Y$ with $|Y \setminus \dom{p}|=|Y \setminus \im{p}|=|Y|$. Hence there exists $p' \in \Sym{Y}$ such that 
\begin{equation} \label{eq:p'}
    \eval{p'}_{\dom{p}}=p \quad \text{ and } \quad  (\im{f} \setminus \dom{p})p'\subseteq Y \setminus \dom{g}.
\end{equation}
By \eqref{lem:total->all/YSym} there exists $p'' \in G \subseteq U$ such that $\eval{p''}_{Y}=p'$. We now show that $fp''g=h$. Note that $ff^{-1}=g^{-1}g$ is the identity on $X$ since $f$ and $g^{-1}$ are total. Moreover, $p''$ extends $p'$ which, in turn, extends $p=f^{-1}hg^{-1}$, and so it follows that 
$$h = f\left(f^{-1}hg^{-1}\right)g = fpg \subseteq fp'g \subseteq fp''g\in U.$$
It remains to show that $x\not \in \dom{h} \implies x \not \in \dom{fp''g}$.
Since $g^{-1}$ is total $\dom{hg^{-1}}=\dom{h}$ and so $\dom{p}=\dom{f^{-1}hg^{-1}}=\dom{h}f$.
Hence, if $x \in X \setminus \dom{h}$, then $(x)f \in \im{f} \setminus \dom{h}f=\im{f} \setminus \dom{p}$. By construction in \eqref{eq:p'}, we now have that $(x)fp'' \in Y \setminus \dom{g}$ and so $x\not \in \dom{fp''g}$, as required. We have shown that $h=fp''g\in U$ and so $f G g=I_X$ since $h$ was arbitrary. In particular, $U=I_X$, since $f,g \in U$ and $G \subseteq U$.
    
For the converse implication, assume that $\mathcal{M}$ is the collection of all maximal subsemigroups of $I_X$ which contain $G$. Clearly conditions \eqref{lem:total->all/G_subset}-\eqref{lem:total->all/M_inverses} hold. To see that \eqref{lem:total->all/in-out-condition} holds let $U$ be a subsemigroup of $I_X$ with $G \subseteq U$ and $U \not \subseteq M$ for all $M \in \mathcal{M}$. As part of the proof of the direct implication above, we have shown that there exist $f$ and $g$ in $I_X$ such that $fGg=I_X$. In particular, $G$ has finite relative rank in $I_X$ and hence so does $U\supseteq G$.  If $U$ were a proper subsemigroup of $I_X$, then, by Lemma \ref{prop:contained_relative_rank}, $U$ would be contained in an element of $\mathcal{M}$. Hence $U=I_X$ and, in particular, $U$ contains an element $f$ with the required properties. 

The final part of the theorem follows from Lemma \ref{lem:inverse_intersection},  since,  as we have shown, 
$G$ has finite relative rank in $I_X$.  
\end{proof}

\subsection{From $I_X$ to $X^X$ and back again}\label{fromIXtoXX}

We will prove the main theorems of this paper by applying Theorem \ref{lem:total->all} to the relevant set 
$\mathcal{M}$ of subsemigroups of $I_X$.  The most involved step in each case is to establish that condition (v) of Theorem \ref{lem:total->all} is satisfied. That is,  we will need to show that certain subsemigroups $U$ of $I_X$ always contain a total element with certain properties. 
As stated in the introduction,  our main theorems are analogues of classification results for maximal subsemigroups of $X^X$.  Moreover,  the proofs of these results for $X^X$ as presented in \cite{maximal_east_mitchell_peresse} involve an analogous key step: to show that certain subsemigroups $V$ of $X^X$ always contain an injective element with certain properties. Of course, the total elements of $I_X$ are precisely the injective elements of $X^X$. 
Our strategy in performing the aforementioned key steps in the proofs of our main theorems will be as follows.

\vspace{\baselineskip}\noindent
Step 1: Given a subsemigroup $U$ of $I_X$, define an analogous subsemigroup $V$ of $X^X$.

\vspace{\baselineskip}\noindent
Step 2: Use results from \cite{maximal_east_mitchell_peresse} to show that $V$ contains an injective map $f$ with certain properties.

\vspace{\baselineskip}\noindent Step 3: Conclude that $f\in U$.

\vspace{\baselineskip}
We now provide the required definitions and results to make Step 1 precise and Step 3 valid. 
An \emph{assignment of transversals} for a subset $V$ of $X^X$ is a function $\Lambda$ from $V$ to the powerset $\mathcal{P}(X)$ of $X$ such that $\Lambda(v)$ is a transversal of $v$ for every $v \in V$.  Let $\mutt{V}{\Lambda}$ denote the set of elements of $X^X$ of the form $v_0v_1\cdots v_n$ such that $v_0, \dots, v_n \in V$ and $\im{v_0\cdots v_{i}}\subseteq \Lambda(v_{i+1})$ for all 
$i\in \{0, \ldots, n-1\}$.  If $u\in I_X$ is non-empty and $v\in X^X$,  then we will say that $v$ is a \emph{minimal transformation extension} of $u$ if
$v|_{\dom{u}}=u$ and $\im{v}=\im{u}$.  Note that if $v$ is a minimal transformation extension of $u$, then
\begin{itemize}
\item $\dom{u}$ is a transversal of $v$;
\item $r(u)=r(v)$, $c(u)=c(v)$, and $d(u)=d(v)$; 
\item $u$ is total $\iff$ $u=v$ $\iff$ $v$ is injective;
\item $\Sigma u \subseteq \Sigma v$ for all $\Sigma\subseteq X$.
\end{itemize}
Let $U\subseteq I_X$.   A \emph{minimal transformation extension of $U$} is a function $\mte \from U\setminus \{\emptyset \} \to X^X$ such that $\mte(u)$ is a minimal transformation extension of $u$ for every non-empty $u\in U$.  We denote the image of $\mte$ by $\mte(U)$. Note that distinct elements of $U$ may have the same minimal transformation extension. An assignment of transversals $\Lambda:\mte(U) \rightarrow \mathcal{P}(X)$ for $\mte(U)$ is called \emph{$\mte$-cohesive} if for every $v\in \mte(U)$, there exists $u \in U$ with $\mte(u)=v$ and $\Lambda(v)=\Lambda(\mte(u))=\dom{u}$. 

\begin{lemma}\label{lem:fromIXtoXX}
Let $X$ be an infinite set, $U \subseteq I_X$,  and $\mte$ a minimal transformation extension of $U$.  If $\Lambda$ is a $\mte$-cohesive assignment of transversals for $\mte(U)$, then every injective element of $\mutt{\mte(U)}{\Lambda}$ is an element of the semigroup generated by $U$. 
\end{lemma}
\begin{proof}
We will show that every injective element of $\mutt{\mte(U)}{\Lambda}$ is an element of the semigroup generated by $U'=\set{u\in U}{\dom{u}=\Lambda(\mte(u))}$. 
By definition, every $f\in \mutt{\mte(U)}{\Lambda}$ is a product of the form 
$$f=\mte(u_0)\mte(u_1) \cdots \mte(u_n)$$ 
for some $u_0, \dots, u_n \in U'$ such that 
\begin{equation}\label{extension_containment}
\im{\mte(u_0)\cdots \mte(u_{i})} \subseteq \dom{u_{i+1}}
\end{equation}
for all $i\in\{0, \dots, n-1\}$. Assume that $f$ is injective.  We will show by induction on $i$ that 
\begin{equation}\label{extension_induction}
\mte(u_0)\cdots \mte(u_{i})=u_0\cdots u_{i}
\end{equation}
for all $i\in \{0, \dots, n\}$. Note that if $a,b \in X^X$ and $ab$ is injective, then $a$ is injective. 
Thus, since $f$ is injective, $\mte(u_0)$ is injective  
and so $\mte(u_0)=u_0$.  So let $0\leq i<n$ and assume that \eqref{extension_induction} holds.  Then $\mte(u_0)\cdots \mte(u_{i})\mte(u_{i+1})=u_0\cdots u_i\mte(u_{i+1})$.  Moreover,   
$\im{u_0\cdots u_i}\subseteq \dom{u_{i+1}}$ by \eqref{extension_containment} and $\mte(u_{i+1})$ agrees with $u_{i+1}$ on $\dom{u_{i+1}}$.  Hence 
$u_0\cdots u_i\mte(u_{i+1})=u_0\cdots u_i u_{i+1}$. We have show that $f$ lies in the semigroup generated by $U'$, as required. 
\end{proof}

\section{The Symmetric Group}

In this section we classify the maximal (inverse) subsemigroups of $I_X$ that contain $\Sym{X}$, when $X$ is an infinite set.

 \begin{theorem} \label{thm:maximal_sym}
Let $X$ be an infinite set.
Then the maximal subsemigroups of $I_X$ which contain $\Sym{X}$ are:
\begin{align*}
S_\mu &= \set{f \in I_X }{ c(f) \geq \mu \text{ or } d(f)<\mu} \\
\inv{S}_\mu &= \set{f \in I_X }{ c(f)<\mu \text{ or } d(f) \geq \mu}
\end{align*}
where $\mu=1$ or $\mu$ is any infinite cardinal with $\mu \leq \card{X}$.  
The maximal inverse subsemigroups of $I_X$ which contain $\Sym{X}$ are 
    \begin{align*}
S_{\mu} \cap \inv{S}_{\mu} = \set{f \in I_X }{\;&(c(f)\geq \mu \text{ and }d(f)\geq \mu) \text{ or }\\ &(c(f)<\mu \text{ and } d(f)<\mu)}
    \end{align*}
over the same values of $\mu$.
\end{theorem}
An alternative, and arguably more natural way, to write $S_1, \inv{S}_1$, and $S_1\cap \inv{S}_1$ from Theorem \ref{thm:maximal_sym} is 
\begin{align*}
S_1&=\Sym{X} \cup \set{f \in I_X }{ c(f)>0};\\
\inv{S}_1&=\Sym{X} \cup \set{f \in I_X }{ d(f)>0};\\
S_1\cap \inv{S}_1 &=\Sym{X} \cup \set{f \in I_X }{ c(f)>0 \text{ and } d(f)>0}.
\end{align*}

The key step in proving Theorem \ref{thm:maximal_sym} is Lemma \ref{lem:constructing_f} below. Lemma \ref{sym_inj} was proved in \cite{maximal_east_mitchell_peresse}, is the analogue of Lemma \ref{lem:constructing_f} for $X^X$, and involves the following $X^X$-analogue of $S_{\mu}$:
\begin{equation}\label{trans_smu_definition}
\trans{S_{\mu}}=\set{f\in X^X }{ c(f) \geq \mu \text{ or } d(f)<\mu}.
\end{equation}
In \cite{maximal_east_mitchell_peresse},  $\trans{S_\mu}$ is denoted by $S_2$ in the case when $\mu=1$ and by $S_4(\mu)$ when $\aleph_0\leq \mu\leq |X|$.

\begin{lemma}[Lemma 6.3 in \cite{maximal_east_mitchell_peresse}] \label{sym_inj}
Let $X$ be an infinite set and let $U$ be a subset of $X^X$, which is not contained in 
$\trans{S_{\mu}}$ (as defined in \eqref{trans_smu_definition})
for $\mu=1$ or any infinite $\mu\leq |X|$,  let $\Lambda$ be any assignment of transversals for $U$,  let $f_0 \in U$ be injective, and let $\kappa$ be an infinite  cardinal such that $\kappa\leq |X|$. If $U$ 
 contains every  $a\in \Sym{X}$ with $\supp{a}\subseteq \im{f_0}$ and $\card{\supp{a}} < \kappa$, then there exists an injective 
 $f_{\kappa}\in \mutt{U}{\Lambda}$ such that $d(f_{\kappa})\geq \kappa$ and $\im{f_{\kappa}}\subseteq \im{f_0}$. 
\end{lemma}

\begin{lemma}\label{lem:constructing_f}
Let $U$ be a subsemigroup of $I_X$ which, for $\mu=1$ and all infinite $\mu\leq |X|$, is not contained in $S_{\mu}$ (as defined in Theorem \ref{thm:maximal_sym}), let $f_0\in U$ be total (i.e. $c(f_0)=0$), and let $\kappa$ be an infinite cardinal with $\kappa \leq |X|$. If $U$ contains every $a \in \Sym{X}$ with $\supp{a} \subseteq \im{f_0}$ and $\supp{a}< \kappa$, then $U$ contains a total element $f_{\kappa}$ with $d(f_{\kappa})\geq \kappa$ and $\im{f_{\kappa}}\subseteq \im{f_0}$.
\end{lemma}
\begin{proof}
Let $U, f_0$, and $\kappa$ satisfy the assumptions of the lemma.  Let $\mte$ be a minimal transformation extension of $U$ (as defined in Section \ref{fromIXtoXX}).  We will now show that $\mte(U)$ satisfies the conditions on Lemma \ref{sym_inj} where $\Lambda$ is a $\mte$-cohesive assignment of transversals for $\mte(U)$.

If $\mu=1$ or $\aleph_0 \leq \mu\leq |X|$ and
$v\in X^X$ is a minimal transformation extension of some $u\in I_X \setminus \{\emptyset\}$,  then,  since minimal transformation extensions preserve collapse and defect, $u\in S_{\mu}$ if and only if $v \in \trans{S_{\mu}}$.  Hence $\mte(U)$ is not contained in any $\trans{S_{\mu}}$ since $U$ is not contained in any $S_{\mu}$.

Furthermore, $U \cap \Sym{X}=\mte(U) \cap \Sym{X}$ and $\mte(f_0)=f_0$. Hence $\mte(U)$ satisfies the conditions of Lemma \ref{sym_inj} and so there exists an injective $f_{\kappa} \in \mutt{\mte(U)}{\Lambda}$ such that $d(f_{\kappa})\geq \kappa$ and $\im{f_{\kappa}}\subseteq \im{f_0}$. By Lemma \ref{lem:fromIXtoXX},  $f_{\kappa}\in U$, as required.  
\end{proof}

Using Theorem \ref{lem:total->all} and Lemma \ref{lem:constructing_f}, it is now straightforward to prove Theorem \ref{thm:maximal_sym}. 

\begin{proof}[Proof of Theorem \ref{thm:maximal_sym}]
Let $Y$ be any moiety of $X$. Then $G=\Sym{X}$ clearly satisfies condition \eqref{lem:total->all/YSym} of Theorem \ref{lem:total->all}.  We will show that the semigroups $S_{\mu}$ and $\inv{S}_{\mu}$ together satisfy the five conditions of Theorem \ref{lem:total->all}.

\vspace{\baselineskip}\noindent
{\bf (i)}
If $f \in \Sym{X}$ and $\mu\geq 1$, then $c(f)=d(f)=0<\mu$. Hence $f$ is contained in every $S_{\mu}$ and $\inv{S}_{\mu}$.

        \vspace{\baselineskip}\noindent
{\bf (ii)} Let $\mu=1$ or $\mu$ be infinite with $\mu\leq |X|$. Clearly, $S_{\mu}$ is a proper subset of $I_X$ since there exist elements $h$ of $I_X$ with $c(h)=0$ and $d(h)=|X|$. 
To show that $S_{\mu}$ is a semigroup, let $f,g \in S_{\mu}$ be arbitrary. If $c(f)\geq \mu$, then $c(fg)\geq \mu$ by  Lemma \ref{lem:defect_properties} part \eqref{lem:defect_properties/defect<collapse} and so $fg \in S_{\mu}$. So we may assume that $c(f) <\mu$ and so $d(f)< \mu$. If $d(g)<\mu$, then $d(fg) \leq d(f) + d(f)$ by Lemma \ref{lem:defect_properties} part \eqref{lem:defect_properties/defect_composition}. Since $\mu=1$ or $\mu$ is infinite, this implies that $d(fg)<\mu$ and so $fg \in S_{\mu}$.

The only remaining case is that $d(f)<\mu$ and $c(g)\geq \mu$. If $\mu$ is infinite, then $c(fg)\geq \mu$ by Lemma \ref{lem:defect_properties} part \eqref{lem:defect_properties/defect<collapse} and so $fg \in S_{\mu}$. Similarly, if $\mu=1$, then $d(f)=0$ and $c(g)\geq \mu=1$ and so \ref{lem:defect_properties} part \eqref{lem:defect_properties/collapse_surjective} implies that $c(fg)=c(f)+c(g)\geq \mu$.        
We have shown that every $S_{\mu}$, and hence every $\inv{S}_{\mu}$, is a proper subsemigroup of $I_X$.
        
\vspace{\baselineskip}\noindent
{\bf (iii)} 
Let $\mu$ and $\nu$ be cardinals that are either $1$ or infinite and at most $|X|$. Then $S_{\mu}\setminus \inv{S}_{\nu}$ contains an element $f$ of $I_X$ satisfying $c(f)=|X|$ and $d(f)=0$. Hence $S_{\mu} \nsubseteq \inv{S}_{\nu}$. 
Now suppose that $\nu \neq \mu$.  If $\nu<\mu$, then let $h_{\mu}\in I_X$ satisfy $c(h_{\mu})=0$ and $d(h_{\mu})=\nu$. On the other hand, if $\mu<\nu$, then let $h_{\mu}$ satisfy $c(h_{\mu})=\mu$ and $d(h_{\mu})=|X|$. Then, in either case, $h_{\mu} \in S_{\mu} \setminus S_{\nu}$ and so $S_{\mu} \nsubseteq S_{\nu}$.        
Since $S_{\mu}$ is not contained in any of the other semigroups, neither is $\inv{S}_{\mu}$.

\vspace{\baselineskip}\noindent
{\bf (iv)} Recalling that $c(h^{-1})=d(h)$ for elements $h$ of $I_X$, it is easy to see that $\inv{S}_{\mu}$ is indeed the inverse of $S_{\mu}$.

        \vspace{\baselineskip}\noindent
{\bf (v)} Let $U$ be a subsemigroup of $I_X$ containing $\Sym{X}$, but which is itself not contained in any of the $S_{\mu}$. Applying Lemma \ref{lem:constructing_f} to $U$ where $f_0$ is the identity function and $\kappa=|X|$, we conclude that $U$ contains a total element $f_{|X|}$ with $d(f_{|X|})=|X|$. Since $Y$ and $\im{f_{|X|}}$ are moieties of $X$, there exists $a\in \Sym{X} \subseteq U$ such that $\left( \im{f_{|X|}}\right)a=Y$. Then $f=f_{|X|}a$ is total and $\im{f}\subseteq Y$, as required. 

\vspace{\baselineskip}
Since all five conditions are satisfied,  it follows from Theorem \ref{lem:total->all} that the set of maximal subsemigroups of $I_X$ containing  $\Sym{X}$ is
$$\set{S_{\mu}, \inv{S}_{\mu}}{ \mu=1 \text{ or }\mu \text{ is an infinite cardinal with }\mu\leq |X|}$$
and that the maximal inverse subsemigroups of $I_X$ containing $\Sym{X}$ are precisely the maximal elements of the set
$$\set{S_{\mu}\cap \inv{S}_{\mu}}{ \mu=1 \text{ or }\mu \text{ is an infinite cardinal with }\mu\leq |X|}.$$
To show that the latter set is an anti-chain under containment,  let $\mu<\nu$ be distinct cardinals that are either $1$ or infinite and at most $|X|$. If $g$ and $h$ are any elements of $I_X$ satisfying $c(g)=\mu$, $d(g)=\nu$, $c(h)=0$, and $d(h)=\mu$, then 
$g \in (S_{\mu}\cap \inv{S}_{\mu}) \setminus (S_{\nu} \cap \inv{S}_{\nu})$ 
and $h \in (S_{\nu}\cap \inv{S}_{\nu})\setminus (S_{\mu}\cap \inv{S}_{\mu})$.
\end{proof}

\section{Pointwise stabilisers of finite sets}

Recall that the \emph{setwise stabiliser} $\setStab{\Sym{X}}{\Sigma}$ and the \emph{pointwise stabiliser} $\pointStab{\Sym{X}}{\Sigma}$ of a subset $\Sigma$ of $X$ are defined by
\begin{align*}
    \setStab{\Sym{X}}{\Sigma} &= \set{f \in \Sym{X} }{ \Sigma f = \Sigma}\\
     \pointStab{\Sym{X}}{\Sigma}&=\set{f \in \Sym{X} }{ (\forall x \in \Sigma) (xf=x) }.
\end{align*}
In this section we classify the maximal subsemigroups and maximal inverse subsemigroups of $I_X$ which contain the pointwise stabiliser of a finite non-empty subset $\Sigma$ of $X$. It turns out that each such maximal (inverse) subsemigroup contains the setwise stabiliser of some non-empty subset $\Gamma$ of $\Sigma$.
Recall that $\fin=\set{f \in I_X}{r(f)<|X|}$.

\begin{theorem} \label{thm:maximal_pointwise}
Let $X$ be an infinite set and $\Sigma$ a non-empty finite subset of $X$.
Then the maximal subsemigroups of $I_X$ which contain the pointwise stabiliser $\pointStab{\Sym{X}}{\Sigma}$ but not $\Sym{X}$ are
\begin{align*}
P_{\Gamma,\mu} =\set{f \in I_X }{\;&c(f) \geq \mu \text{ or } \Gamma \nsubseteq \dom{f} \text{ or } \\
&(\Gamma f = \Gamma \text{ and } d(f) < \mu)} \cup \fin \\
\inv{P}_{\Gamma,\mu} =\set{f \in I_X }{\;&d(f) \geq \mu \text{ or } \Gamma \nsubseteq \im{f} \text{ or } \\
&(\Gamma f = \Gamma \text{ and } c(f) < \mu)} \cup \fin
\end{align*}
where $\Gamma$ is a non-empty subset of $\Sigma$ and $\mu$ is an infinite cardinal with $\mu \leq \card{X}^+$.
The maximal inverse subsemigroups of $I_X$ containing $\pointStab{\Sym{X}}{\Sigma}$ but not $\Sym{X}$ are
\begin{align*}
 P_{{\Gamma}, \mu} \cap \inv{P}_{\Gamma, \mu}
= \lbrace f \in I_X :\;
&(\Gamma \nsubseteq \dom{f} \text{ and } \Gamma \nsubseteq \im{f})    \text{ or } \\
&(\Gamma \nsubseteq \dom{f}   \text{ and } d(f)\geq \mu)     \text{ or } \\
&( c(f)\geq \mu \text{ and } \Gamma \nsubseteq \im{f})     \text{ or } \\
&(c(f)\geq \mu \text{ and }d(f)\geq \mu)    \text{ or } \\
&(\Gamma f = \Gamma \text{ and } c(f)+d(f) < \mu)\rbrace \cup \fin
\end{align*}
over the same values of $\Gamma$ and $\mu$.
\end{theorem}
Note that when $\mu\leq |X|$, then the ``$\cup\,\fin$'' could be omitted from the definition of the semigroups in Theorem \ref{thm:maximal_pointwise} since every element of $I_X$ with rank strictly less than $|X|$ has both collapse and defect at least $\mu$. We can also simplify the definitions in the case $\mu=|X|^+$ as follows:
\begin{align*}
P_{\Gamma, |X|^+}=\set{f \in I_X }{&\;\Gamma \not \subseteq \dom{f} \text{ or } \Gamma f=\Gamma} \cup \fin;\\
\inv{P}_{\Gamma, |X|^+}=\set{f \in I_X }{&\;\Gamma \not \subseteq \im{f} \text{ or } \Gamma f=\Gamma} \cup \fin;\\
 P_{{\Gamma}, |X|^+} \cap \inv{P}_{\Gamma, |X|^+}= \set{f \in I_X}
 {&\;(\Gamma \nsubseteq \dom{f} \text{ and } \Gamma \nsubseteq \im{f})    \text{ or } \\
&\;(\Gamma f = \Gamma)} \cup \fin.
\end{align*}

To prove Theorem \ref{thm:maximal_pointwise}, we will once again make use of results about analogous maximal subsemigroups of $X^X$ from \cite{maximal_east_mitchell_peresse}. 
The following are the $X^X$-analogues of $\fin$ and $P_{\Gamma, \mu}$:

\begin{nalign}\label{trans_stab_definition}
\trans{\fin}= \set{f \in X^X}{&\;r(f)<|X|};\\
\trans{P_{\Gamma, \mu}}= \set{f \in X^X }{ &\;c(f) \geq \mu \text{ or } |\Gamma f| < |\Gamma| \text{ or }\\ 
&(\Gamma f = \Gamma \text{ and } d(f) < \mu)} \cup \trans{\fin}.
\end{nalign} 
The result about $\trans{P_{\Gamma, \mu}}$ relevant to us is \cite[Lemma 7.2]{maximal_east_mitchell_peresse} (where $\trans{P_{\Gamma, \mu}}$ is denoted by $F_2(\Gamma,  \mu)$). However, there is the following technical problem with the way that \cite[Lemma 7.2]{maximal_east_mitchell_peresse} is stated. The lemma demands that a certain subset $U$ of $X^X$ be given an assignment of transversals $\Lambda$ such that $\Gamma \subseteq \Lambda(u)$ for all $u \in U \setminus \trans{P_{\Gamma, \mu}}$. Such a $\Lambda$ certainly exists for any fixed $\Gamma$ but there may not be a single $\Lambda$ satisfying this property for every non-empty subset $\Gamma$ of $\Sigma$. The proof of Lemma 7.2 given in \cite{maximal_east_mitchell_peresse} does prove a slightly different statement which is sufficient for the intended applications here and in \cite{maximal_east_mitchell_peresse}. It is this latter statement that we now give; together with an apology by the second author for the part that they played as an author of \cite{maximal_east_mitchell_peresse} in creating this confusion in the first place. 

\begin{lemma}[implied by the proof of Lemma 7.2 in \cite{maximal_east_mitchell_peresse}]\label{finite_stab_inj}
Let $X$ be an infinite set, $\Sigma$ a finite subset of $X$, and $U$ a subset of $X^X$ which contains $\Sym{X}_{(\Sigma)}$ and is not contained in $\trans{S_1}$ (as defined in \eqref{trans_smu_definition}). If $f \in U$ is injective such that $\im{f}\cap \Sigma = \Gamma \neq \emptyset$, $U$ is not contained in $\trans{P_{\Gamma, \mu}}$ (as defined in \eqref{trans_stab_definition})
for any infinite cardinal $\mu \leq |X|^+$, and $\Lambda$ is any assignment of transversals for $U$ such that $\Gamma \subseteq \Lambda(u)$ for all $u \in U \setminus \trans{P_{\Gamma, \mu}}$, then there exists an injective $f'\in \mutt{U}{\Lambda}$ such that $\im{f'}\cap \Sigma \subsetneq \Gamma$.
\end{lemma}

\begin{lemma}\label{finite_stab_total}
Let $X$ be an infinite set, $\Sigma$ a finite subset of $X$,  and $U$ a subset of $I_X$ containing $\Sym{X}_{(\Sigma)}$ but which is not 
contained in $S_1$ (as defined in Theorem \ref{thm:maximal_sym}) or $P_{\Gamma, \mu}$ (as defined in Theorem \ref{thm:maximal_pointwise}) for any non-empty subset $\Gamma$ of $\Sigma$ and 
any infinite cardinal $\mu \leq |X|^+$.  Then there exists a total $f$ in the semigroup generated by $U$ such that $\im{f} \cap \Sigma=\emptyset$.
\end{lemma}
\begin{proof}
Let $\Sigma$ and $U$ satisfy the conditions of the lemma. Since $\Sigma$ is finite and the identity function on $X$ lies in $U$, it suffices to show that for every total $f\in U$ with $\im{f}\cap \Sigma\neq \emptyset$ , there exists a total $f'\in U$ such that $\im{f'} \cap \Sigma \subsetneq \im{f}\cap \Sigma$. So let $f\in U$ be total with $\im{f} \cap \Sigma = \Gamma \neq \emptyset$.

We will show that the required $f'$ may be found in the semigroup generated by $U'=\set{u \in U}{\Gamma \subseteq \dom{u}}$. Note that $\set{u\in U}{\Gamma \not \subseteq \dom{u}}$ contains no total elements and is contained in $S_1$ and $P_{\Gamma, \mu}$ for every $\mu$. Thus $U'$ still contains $f$ and $\Sym{X}_{(\Sigma)}$ and is not contained in $S_1$ or $P_{\Gamma, \mu}$ for any infinite $\mu\leq |X|^+$.

Let $\mte$ be a minimal transformation extension of $U'$ as defined in Section \ref{fromIXtoXX} and let $\Lambda$ be a $\mte$-cohesive assignment of transversals for $\mte(U')$. 
We will show that $\mte(U')$ satisfies the conditions of Lemma \ref{finite_stab_inj}. 
Since $\mte(u)=u$ whenever $u$ is total and $U'$ contains $f$ and $\Sym{X}_{(\Sigma)}$, it follows that $\mte(U')$ also contains $f$ and $\Sym{X}_{(\Sigma)}$. 

Recall that minimal transformation extensions preserve rank, collapse and defect. Hence $u \in S_1 \iff \mte(u) \in \trans{S_1}$ and so $\mte(U') \not \subseteq \trans{S_1}$, since $U' \not \subseteq S_1$.
Furthermore, if $u\in U'$, then $\Gamma \subseteq \dom{u}$ and so $|\Gamma \mte(u)|=|\Gamma|$ and $\Gamma u=\Gamma \mte(u)$. Hence $u \in P_{\Gamma, \mu} \iff \mte(u) \in \trans{P_{\Gamma, \mu}}$ for every $u\in U'$.
Since $U'$ is not contained in any $P_{\Gamma, \mu}$, it follows that $\mte(U')$ is not contained in any $\trans{P_{\Gamma, \mu}}$. We have shown that $\mte(U')$  satisfies the conditions of Lemma \ref{finite_stab_inj}. Thus there exists an injective $f'\in \mutt{\mte(U')}{\Lambda}$ such that $\im{f}\cap \Sigma\subsetneq \Gamma$ and,  by Lemma \ref{lem:fromIXtoXX},  $f'$ lies in the semigroup generated by $U'$, as required.  
\end{proof}

\begin{proof}[Proof of Theorem \ref{thm:maximal_pointwise}.]
If $Y$ is any moiety of $X$ which is disjoint from $\Sigma$, then $G=\pointStab{\Sym{X}}{\Sigma}$ satisfies condition \eqref{lem:total->all/YSym} of Theorem \ref{lem:total->all}. We will now show that the semigroups $P_{\Gamma, \mu}$ and $\inv{P}_{\Gamma, \mu}$ together with the maximal semigroups  containing $\Sym{X}$ (see Theorem \ref{thm:maximal_sym}) satisfy the five conditions of Theorem \ref{lem:total->all}.

\vspace{\baselineskip}\noindent
{\bf (i)} Let $\emptyset \neq \Gamma \subseteq \Sigma$ and $\mu \leq |X|^+$ be an infinite cardinal. If $f \in \pointStab{\Sym{X}}{\Sigma}$, then $\Gamma f=\Gamma$ and $c(f)=d(f)=0<\mu$. Thus $f\in P_{\Gamma, \mu} \cap \inv{P}_{\Gamma, \mu}$.
        
\vspace{\baselineskip}\noindent
{\bf(ii)}
It suffices to show that every $P_{\Gamma,\mu}$ is a proper subsemigroup of $I_X$.
Fix a non-empty finite subset $\Gamma$ of $\Sigma$ and an infinite cardinal $\mu \leq |X|^+$. 
        If $f\in \Sym{X}$ does not set-wise stabilise $\Gamma$, then $f \not \in P_{\Gamma,\mu}$. In particular, $P_{\Gamma, \mu}$ is a proper subset of $I_X$. It remains to show that $st \in P_{\Gamma, \mu}$ for arbitrary $s,t \in P_{\Gamma, \mu}$. Since $\fin$ is an ideal of $I_X$ we may assume that 
        $s,t \in  \set{f \in I_X }{ c(f) \geq \mu \text{ or } \Gamma \nsubseteq \dom{f} \text{ or } (\Gamma f = \Gamma \text{ and } d(f) < \mu)}$.
If $c(s)\geq \mu$, then $c(st) \geq c(s)\geq \mu$ by Lemma \ref{lem:defect_properties} part \eqref{lem:defect_properties/collapse_composition}. Similarly, if $\Gamma \not \subseteq \dom{s}$, then $\Gamma \not \subseteq \dom{st}$. So assume that $\Gamma s=\Gamma$ and $d(s)<\mu$. If $c(t) \geq \mu$,  then $c(st)\geq \mu$ by Lemma \ref{lem:defect_properties} part \eqref{lem:defect_properties/defect<collapse} and if $\Gamma \not \subseteq \dom{t}$, then $\Gamma \not \subseteq \dom{st}$. Finally, if $\Gamma t=\Gamma$ and $d(t)<\mu$, then $\Gamma st=\Gamma t=\Gamma$ and $d(st)\leq d(s)+d(t)<\mu+\mu=\mu$ by Lemma \ref{lem:defect_properties} part \eqref{lem:defect_properties/defect_composition}.
        
\vspace{\baselineskip}\noindent
{\bf(iii)} 
We fix a non-empty subset $\Gamma$ of $\Sigma$ and an infinite $\mu\leq |X|^+$ and show that $P_{\Gamma, \mu}$ is not contained in any of the other semigroups from this theorem or Theorem \ref{thm:maximal_sym}. It will then follow that the same is true for $\inv{P}_{\Gamma, \mu}$. In some instances we will actually show that $P_{\Gamma, \mu} \cap \inv{P}_{\Gamma, \mu}$ is not contained in certain semigroups as this will be needed later in this proof when we consider maximal inverse subsemigroups.

If $f_1\in I_X$ satisfies $\Gamma f_1=\Gamma$, $d(f_1)=1$, and $c(f_1)=0$, then $f_1 \in \left(P_{\Gamma, \mu} \cap \inv{P}_{\Gamma, \mu}\right) \setminus S_1$.
If $f_2\in I_X$ has domain $X \setminus \Sigma$ and satisfies $d(f_2)=|X|$,  then $f_2\in\left(P_{\Gamma, \mu} \cap \inv{P}_{\Gamma, \mu}\right)\setminus S_{\nu}$ for every infinite $\nu\leq |X|$.  
If $f_3\in I_X$ satisfies $c(f_3)=|X|$, $\Gamma \not \subseteq \dom{f_3}$, and $d(f_3)=0$, then $f_3 \in P_{\Gamma, \mu}$ but $f_3 \not \in \inv{S}_{\nu}$ for $\nu=1$ or any infinite $\nu\leq |X|$.
Fix another infinite $\nu \leq |X|^+$ and non-empty subset $\Delta$ of $\Sigma$.  
If $f_4\in I_X$ has domain $X \setminus \Sigma$ and image $X$, then $f_4\in P_{\Gamma, \mu}$ but $f_4\not \in \inv{P}_{\Delta, \nu}$.  

Finally, we will find $g_{\Delta, \nu}\in \left(P_{\Gamma, \mu} \cap \inv{P}_{\Gamma, \mu}\right) \setminus P_{\Delta, \nu}$ for every non-empty $\Delta \subseteq \Sigma$ and infinite $\nu \leq |X|^+$ such that $(\Delta, \nu) \neq (\Gamma, \mu)$.  If $\Gamma \not \subseteq \Delta$, then let $g_{\Delta, \nu}$ be any element of $I_X$ with domain $X \setminus (\Gamma \setminus \Delta)$ and image $X \setminus \Sigma$.  If $\Delta \not \subseteq \Gamma$, then let $g_{\Delta, \nu}$ be an element of $\Sym{X}$ which stabilises $\Gamma$ but not $\Delta$.  If $\mu<\nu$, then let $g_{\Delta, \nu}$ satisfy $c(g_{\Delta, \nu})=\mu$, $\Sigma \subseteq \dom{g_{\Delta, \nu}}$ and $\im{g_{\Delta, \nu}} \cap \Sigma = \emptyset$.  If $\nu<\mu$, then let $g_{\Delta, \nu}$ be total (i.e. $c(g_{\Delta, \nu})=0$), fix $\Sigma$ pointwise and satisfy $d(g_{\Delta, \nu})=\nu$.

\vspace{\baselineskip}\noindent
{\bf(iv)} 
$\inv{P}_{\Gamma,  \mu}$ is indeed the inverse of $P_{\Gamma, \mu}$, since, for all $f\in I_X$, $c(f) = d(\inv{f})$, $\im{f} = \dom{\inv{f}}$, and if $\Gamma$ is finite,  then $\Gamma f = \Gamma \iff \Gamma \inv{f} = \Gamma$.

\vspace{\baselineskip}\noindent
{\bf(v)}  Let $U$ be a subsemigroup of $I_X$ containing $\pointStab{\Sym{X}}{\Sigma}$, but which is itself not contained in $S_{\nu}$ for $\nu=1$ or any infinite $\nu\leq |X|$ nor contained in $P_{\Gamma, \mu}$ for any infinite $\mu\leq |X|^+$ and any non-empty $\Gamma \subseteq \Sigma$.
By Lemma \ref{finite_stab_total},  there exists a total $f\in U$ with $\im{f}\cap \Sigma=\emptyset$.
Applying Lemma \ref{lem:constructing_f} to $U$ with $f_0=f$ and $\kappa=|X|$,  we obtain a total $f_{|X|} \in U$ with $d(f_{|X|})=|X|$ and $\im{f_{|X|}} \subseteq \im{f} \subseteq X \setminus \Sigma$. If $a \in \pointStab{\Sym{X}}{\Sigma} \subseteq U$ takes $\im{f_{|X|}}$ to $Y$, then $f=f_{|X|}a$ is the required total element with $\im{f}=Y$.

\vspace{\baselineskip} It now follows from Theorem \ref{lem:total->all} that the maximal subsemigroups of $I_X$ containing 
$\pointStab{\Sym{X}}{\Sigma}$ are $P_{\Gamma, \mu}$, $\inv{P}_{\Gamma, \mu}$, $S_{\lambda}$,  and $\inv{S}_{\lambda}$ over all non-empty $\Gamma \subseteq \Sigma$,  infinite $\mu\leq |X|^+$,  and $\lambda=1$ or $\aleph_0\leq \lambda \leq |X|$.  Thus,  the maximal subsemigroups of $I_X$ containing 
$\pointStab{\Sym{X}}{\Sigma}$ but not $\Sym{X}$ are the semigroups $P_{\Gamma, \mu}$ and $\inv{P}_{\Gamma, \mu}$,  as required.  

Moreover,  the maximal inverse subsemigroups of $I_X$ containing $\pointStab{\Sym{X}}{\Sigma}$ are the maximal elements of the corresponding set of intersections $P_{\Gamma, \mu}\cap \inv{P}_{\Gamma, \mu}$ and $S_{\lambda}\cap \inv{S}_{\lambda}$.  It only remains to fix a non-empty finite subset $\Gamma$ of $X$ and an infinite cardinal $\mu\leq |X|^+$ and show that $P_{\Gamma, \mu}\cap \inv{P}_{\Gamma, \mu}$ is maximal in this set of intersections. But we have already done this! In the proof of condition (iii) above, we found $f_1, f_2, g_{\Delta, \nu} \in P_{\Gamma, \mu}\cap \inv{P}_{\Gamma, \mu}$ such that $f_1\not \in S_1 \supseteq S_1 \cap \inv{S}_1$, $f_2 \not \in S_{\nu} \supseteq S_{\nu}\cap \inv{S}_{\nu}$ for every infinite $\nu \leq |X|^+$, and $g_{\Delta, \nu} \not \in P_{\Delta, \nu} \supseteq P_{\Delta, \nu} \cap \inv{P}_{\Delta, \nu}$ for every $(\Delta, \nu) \neq (\Gamma, \mu)$.
\end{proof}

We finish this section with an application of Theorems \ref{thm:maximal_sym} and \ref{thm:maximal_pointwise}.
Recall that an element $e$ of a semigroup $S$ is an \emph{idempotent} if $e^2=e$. The idempotents of $I_X$ are precisely the partial identities $\set{(y,y)}{y\in Y}$ of subsets $Y$ of $X$.

\begin{corollary}\label{cor:intersection}
    Let $X$ be an infinite set. The intersection of all maximal subsemigroups of $I_X$ coincides with the intersection of all maximal inverse subsemigroups of $I_X$ and equals $\fin \cup E_X$, where
    $$\fin=\set{f\in I_X}{r(f)<|X|} \text{ and } E_X = \set{e \in I_X}{e^2 = e}.$$
   
\end{corollary}
\begin{proof}

    Let $M$ be a maximal subsemigroup or maximal inverse subsemigroup of $I_X$. We start by showing that $\fin \cup E_X \subseteq M$. 
    Note that the maximal subsemigroups and maximal inverse subsemigroups of $I_X$ which contain $\Sym{X}$, as classified in Theorem \ref{thm:maximal_sym}, do contain $\fin \cup E_X$. So we may assume that there exists a permutation $f \in \Sym{X} \setminus M$. Aiming for a contradiction, assume that $\fin \cup E_X \not \subseteq M$ and let $g \in \left( \fin \cup E_X\right) \setminus M$. Then $f$ may be written as a product of elements of $M$ together with $g$ (and  $g^{-1}$ in the case that $M$ is a maximal inverse subsemigroup). Since $\fin$ is an ideal of $I_X$ and $f \not \in \fin \cup M$, it follows that $g \in E_X$ and so $g=g^{-1}$. Furthermore, it is easy to see that the identity on $X$ is an element of $M$ and so we may write 
    $$f=m_0gm_1\cdots m_{n-1}gm_n$$
    for some elements $m_0, \dots, m_n \in M$. Since $g$ is an identity on some subset of $X$, it follows that $m_0gm_1\cdots m_{n-1}gm_n$ is a restriction of $m_0m_1\cdots m_n$. But $f$ is total and so $m_0m_1\cdots m_n=f$, contradicting the fact that $f \not \in M$. Hence $\fin \cup E_X \subseteq M$, as required. 

    It remains to show that for every $h \in I_X \setminus \left( \fin \cup E_X\right)$, there exists a maximal subsemigroup and a maximal inverse subsemigroup of $I_X$ which does not contain $h$. Since $h \not \in E_X$, there exists $(x,y) \in h$ with $x\neq y$. Note that if $\Gamma=\{x\}$, then $h \not \in P_{{\Gamma}, |X|^+}$, which is a maximal subsemigroup of $I_X$ by Theorem \ref{thm:maximal_pointwise}. Moreover, $h$ does not lie in the maximal inverse subsemigroup $P_{{\Gamma}, |X|^+} \cap \inv{P}_{\Gamma, |X|^+}$.
\end{proof}

\section{Stabilisers of ultrafilters}\label{section:ultrafilters}

In this section we classify the maximal (inverse) subsemigroups of $I_X$ containing the stabiliser of an ultrafilter (as defined in Section \ref{section:sets}). 
Recall that the \emph{stabiliser} of a filter $\filter$ on $X$ is the group
$$\setStab{\Sym{X}}{\filter}=\set{f \in \Sym{X}}{(\forall \Sigma \subseteq X)(\Sigma \in \filter \iff \Sigma f \in \filter)}.$$
It was shown in \cite[Theorem 6.4]{subgroups_macpherson_neumann} that if $\filter$ is an ultrafilter, then $\setStab{\Sym{X}}{\filter}$ is a maximal subgroup of $\Sym{X}$ and
\begin{equation}\label{ultrafilter_stabiliser_as_union}
\setStab{\Sym{X}}{\filter}=\bigcup_{\Sigma \in \filter} \Sym{X}_{(\Sigma)}.
\end{equation}

\begin{theorem} \label{thm:maximal_ultrafilter}
Let $X$ be an infinite set, $\filter$ a non-principal ultrafilter on $X$,  and $\min({\filter})$ the least cardinality of an element of $\filter$.
Then the maximal subsemigroups of $I_X$ which contain $\setStab{\Sym{X}}{\filter}$ but not $\Sym{X}$ are
\begin{align*}
    V_{\filter, \mu} =\set{f \in I_X }{&\;c(f)\geq \mu \text{ or } \dom{f} \not \in \filter \text{ or } \\
    &((\forall \Sigma \subseteq X)(\Sigma \in \filter \iff \Sigma f \in \filter) \\
    &\text{ and } d(f)<\mu)} \cup \fin\\
    \inv{V}_{\filter, \mu} = \set{f \in I_X }{&\;d(f)\geq \mu \text{ or }
    \im{f} \not \in \filter  \text{ or } \\
    &((\forall \Sigma \subseteq X)(\Sigma \in \filter \iff \Sigma f \in \filter)\\
    &\text{ and }c(f)<\mu)}\cup \fin
\end{align*}
for any cardinal $\mu$ with $\min(\filter)<\mu \leq |X|^+$.
The maximal inverse subsemigroups of $I_X$ which contain $\setStab{\Sym{X}}{\filter}$ but not $\Sym{X}$ are
\begin{align*}
    V_{\filter, \mu} \cap \inv{V}_{\filter, \mu} = \set{f \in I_X }{ 
    \;&(c(f) \geq \mu \text{ and }d(f) \geq \mu) \text{ or } \\
    &(c(f) \geq \mu \text{ and }\im{f} \not \in \filter) \text{ or } \\
    &(\dom{f} \not \in \filter \text{ and }d(f) \geq \mu) \text{ or } \\
    &(\dom{f} \not \in \filter \text{ and }\im{f} \not \in \filter)  \text{ or } \\ 
    &((\forall \Sigma \subseteq X)(\Sigma f \in \filter \iff \Sigma \in \filter)\\
    & \text{ and } c(f) + d(f)<\mu)} \cup \fin
\end{align*}
over the same values of $\mu$.
\end{theorem}

When $\mu\leq |X|$, then the ``$\cup\,\fin$" may be omitted from the definitions of $V_{\filter, \mu}$ and $\inv{V}_{\filter, \mu}$. In the case when $\mu=|X|^+$ (which is the only case when $\min(\filter)=|X|$, i.e. when the filter $\filter$ is uniform) we may more simply write:
\begin{align*}
    V_{\filter, |X|^+} = \set{f \in I_X }{&\; \dom{f} \not \in \filter \text{ or }\\
    &(\forall \Sigma \subseteq X)(\Sigma \in \filter \iff \Sigma f \in \filter)} \cup \fin\\
    \inv{V}_{\filter, |X|^+} = \set{f \in I_X }{&\; \im{f} \not \in \filter \text{ or }\\
    &(\forall \Sigma \subseteq X)(\Sigma \in \filter \iff \Sigma f \in \filter)}\cup \fin\\
    V_{\filter, |X|^+} \cap \inv{V}_{\filter, |X|^+} = \set{f \in I_X }{&\; (\dom{f} \not \in \filter \text{ and }\im{f} \not \in \filter)  \text{ or } \\ 
    &\;(\forall \Sigma \subseteq X)(\Sigma f \in \filter \iff \Sigma \in \filter)}\cup \fin.
\end{align*}

For any value of $\mu$, we can alternatively write the maximal semigroups in Theorem \ref{thm:maximal_ultrafilter} in a way more in line with \cite{maximal_east_mitchell_peresse}:
\begin{nalign} \label{filter_alternative_definition}
    V_{\filter, \mu} =\set{f \in I_X }{&\;c(f)\geq \mu \text{ or } \dom{f} \not \in \filter \text{ or } 
    \\
    &((\forall \Sigma \in \filter)(\Sigma f \in \filter) \text{ and }d(f)<\mu)} \cup \fin  \\
    \inv{V}_{\filter, \mu} = \set{f \in I_X }{&\;d(f)\geq \mu \text{ or }
    \im{f} \not \in \filter  \text{ or } 
    \\
    &((\forall \Sigma \not \in \filter)(\Sigma f \not \in \filter)  \text{ and } c(f)<\mu)}\cup \fin. 
\end{nalign}
To see that these are indeed equivalent definitions, first note that the sets defined in Theorem \ref{thm:maximal_ultrafilter} are clearly subsets of the corresponding sets defined
in \eqref{filter_alternative_definition}. On the other hand, if $f\in I_X$ satisfies $(\forall \Sigma \in \filter)(\Sigma f \in \filter)$ and $\Gamma \not \in \filter$,  then $X \setminus \Gamma \in \filter$ and so $(X \setminus \Gamma) f \in \filter$. Thus $\Gamma f \subseteq X \setminus (X \setminus \Gamma)f \not \in \filter$ and so $f$ satisfies $(\forall \Sigma \subseteq X)(\Sigma \in \filter \iff \Sigma f \in \filter)$. It follows that the definition of $V_{\filter, \mu}$ from \eqref{filter_alternative_definition} agrees with the one in Theorem \ref{thm:maximal_ultrafilter}. To show that the definitions of $\inv{V}_{\filter, \mu}$ agree, let $f \in I_X$ satisfy $(\forall \Sigma \not \in \filter)(\Sigma f \not \in \filter)$. We may assume that $\im{f} \in \filter$, since otherwise $f\in \inv{V}_{\filter, \mu}$ under either definition and we are done. If $\Sigma \in \filter$, then $(X \setminus \Sigma) f \not \in \filter$ and so $\Sigma f=\im{f} \setminus (X \setminus \Sigma) f \in \filter$ since $\filter$ is an ultrafilter. Hence $f$ satisfies $(\forall \Sigma \subseteq X)(\Sigma \in \filter \iff \Sigma f \in \filter)$.

The analogue of $V_{\filter, \mu}$ in the full transformation semigroup $X^X$ is
\begin{nalign}\label{trans_filter_definition}
\trans{V_{\filter, \mu}} =\set{f \in I_X }{&\;c(f)\geq \mu \text{ or }
(\forall \Sigma \in \filter)(c(f|_{\Sigma})>0) \text{ or }\\
 &(d(f)<\mu \text{ and }(\forall \Sigma \in \filter)(\Sigma f \in \filter))} \cup \trans{\fin} 
\end{nalign}

In \cite{maximal_east_mitchell_peresse} $\trans{V_{\filter, \mu}}$ is denoted $U_2(\Gamma, \mu)$ and the following result was proved.

\begin{lemma}[Lemma 8.1 in \cite{maximal_east_mitchell_peresse}]\label{ultra_inj}
Let $\filter$ be a non-principal ultrafilter on an infinite set $X$ and let $\min(\filter)$ be the least cardinality of an element of $\filter$.  Let $U$ be a subset of $X^X$ containing the stabiliser $\Sym{X}_{\{\mathcal{F}\}}$ of $\mathcal{F}$ but which is not 
contained in $\trans{V_{\mathcal{F}, \mu}}$ (as defined in \eqref{trans_filter_definition}), $\trans{S_1}$, or $\trans{S_{\nu}}$ (as defined in \eqref{trans_smu_definition}) for any cardinals $\mu,\nu$ such that 
$\aleph_0\leq\nu\leq \min(\mathcal{F})< \mu\leq |X|^+$, and let $\Lambda$ be any 
assignment of transversals for $U$ such that $\Lambda(u)\in \mathcal{F}$ for all $u\in U\setminus \trans{V_{\mathcal{F}, \mu}}$. Then there exists an 
injective $f\in \mutt{U}{\Lambda}$ such that $\im{f} \not\in \mathcal{F}$. 
\end{lemma}

We now prove the $I_X$-analogue of Lemma \ref{ultra_inj}.

\begin{lemma}\label{ultra_inj_ix}
Let $\filter$ be a non-principal ultrafilter on an infinite set $X$ and let $\min(\filter)$ be the least cardinality of an element of $\filter$.  Let $U$ be a subset of $I_X$ containing the stabiliser $\Sym{X}_{\{\mathcal{F}\}}$ of $\mathcal{F}$ but which is not 
contained in ${V_{\mathcal{F}, \mu}}$ (as defined in Theorem \ref{thm:maximal_ultrafilter}), ${S_1}$, or ${S_{\nu}}$ (as defined in Theorem \ref{thm:maximal_sym}) for any cardinals $\mu,\nu$ such that 
$\aleph_0\leq\nu\leq \min(\mathcal{F})< \mu\leq |X|^+$. Then there exists a total $f$ in the semigroup generated by $U$ such that $\im{f} \not\in \mathcal{F}$. 
\end{lemma}

\begin{proof}
Let $U'=\set{u \in U}{\dom{u} \in \filter}$, let $\mte$ be a minimal transformation extension of $U'$ as defined in Section \ref{fromIXtoXX} and let $\Lambda$ be a $\mte$-cohesive assignment of transversals. By Lemma \ref{lem:fromIXtoXX}, it suffices to show that $\mte(U')$ and $\Lambda$ satisfy the conditions of Lemma \ref{ultra_inj}.
First, since $\dom{u}\in \filter$ for every $u \in U'$, every $v \in \mte(U')$ is assigned a transversal $\Lambda(v)$ in $\filter$, and so $\Lambda$ has the required property on $U'$.

Now we verify the conditions on $\mte(U')$. Note that the complement $U''=\set{u \in U}{\dom{u} \not \in \filter}$ of $U'$ in $U$ contains no permutations and is contained in $V_{\filter, \mu}$ for every infinite $\mu \leq |X|^+$. Hence $U'$ contains $\Sym{X}_{\{\mathcal{F}\}}$ and is not contained in any $V_{\filter, \mu}$. Moreover, if $u \in U''$, then $X \setminus \dom{u} \in \filter$ and so $c(u)=|X \setminus \dom{u}| \geq \min(\filter)$. Hence $U'' \subseteq S_{\nu}$ for all $\nu \leq \min(\filter)$ and so $U'$ is not contained in any such $S_{\nu}$.

Just as in the proofs of Lemmas \ref{finite_stab_total} and \ref{lem:constructing_f}, it follows that $\mte(U')$ contains $\Sym{X}_{\{\mathcal{F}\}}$ and is not contained in $\trans{S_{\nu}}$ for all $\nu\leq \min(\filter)$. Finally, let $\min(\filter) < \mu \leq |X|^+$ and $u \in U' \setminus V_{\filter, \mu}$. It follows from the alternative definition of $V_{\filter, \mu}$ given in \eqref{filter_alternative_definition} that $c(u)<\mu$, $\dom{u}\in \filter$ and either $d(u)\geq \mu$ or there exists $\Sigma \in \filter$ such that $\Sigma u \not \in \filter$.
Thus $c(\mte(u))=c(u)<\mu$ and $c(\mte(u)|_{\Sigma'})=0$ where $\Sigma'=\dom{u} \in \filter$.  If $d(u)\geq \mu$, then $d(\mte(u))=d(u)\geq \mu$ and so $\mte(u) \in \mte(U')\setminus \trans{V_{\filter, \mu}}$ and we are done. Otherwise, there exists $\Sigma \in \filter$ such that $\Sigma u \not \in \filter$. Then $\Sigma''=\Sigma \cap \dom{u} \in \filter$ and $\Sigma'' \mte(u)=\Sigma''u \subseteq \Sigma u \not \in \filter$. We again conclude that $\mte(u) \in \mte(U')\setminus \trans{V_{\filter, \mu}}$.
\end{proof}

\begin{proof}[Proof of Theorem \ref{thm:maximal_ultrafilter}.]
Let $Y$ be any moiety of $X$ such that $Y \not \in \filter$. By \eqref{ultrafilter_stabiliser_as_union},  $G=\Sym{X}_{\{\mathcal{F}\}}$ contains the pointwise stabiliser of $X \setminus Y \in \filter$ and so $G$ satisfies condition \eqref{lem:total->all/YSym} of Theorem \ref{lem:total->all}. Hence it suffices to show that the collection of semigroups $V_{\filter, \mu}$ and $\inv{V}_{\filter, \mu}$ over all cardinals $\mu$ with $\min(\filter) < \mu \leq |X|^+$ together with the maximal semigroups containing $\Sym{X}$ classified in Theorem \ref{thm:maximal_sym} satisfy the five conditions of Theorem \ref{lem:total->all}.

\vspace{\baselineskip}\noindent
{\bf (i)} If $f \in \setStab{\Sym{X}}{\filter}$,  then $c(f)=d(f)=0$ and $\Sigma \in \filter \iff \Sigma f \in \filter$ for every $\Sigma \subseteq X$.  Thus $f$ lies in every $V_{\filter, \mu}$ and $\inv{V}_{\filter, \mu}$.

\vspace{\baselineskip}\noindent
{\bf (ii)} Let $\mu$ be a cardinal with $\min(\filter)<\mu\leq |X|^+$.  Note that $V_{\filter,\mu}\cap\Sym{X}=\setStab{\Sym{X}}{\filter}$.  In particular,  $V_{\filter,\mu}$ and $\inv{V}_{\filter, \mu}$ are proper subsets of $I_X$.

Let $f,g \in V_{\filter,\mu}$.  We need to show that $fg \in V_{\filter,\mu}$. We may assume that $r(f)=r(g)=|X|$ since $\fin$ is an ideal of $I_X$.
If $c(f)\geq \mu$,  then $c(fg)\geq c(f)\geq \mu$ by Lemma \ref{lem:defect_properties} part \eqref{lem:defect_properties/collapse_composition}.  Similarly,  if $\dom{f} \not \in \filter$,  then $\dom{fg}\not \in \filter$,  since $\dom{fg}\subseteq \dom{f}$ and filters are closed under taking supersets.  Hence $fg \in V_{\filter,\mu}$ in both cases.  So we may assume that $f$ satisfies $c(f)<\mu$, $\dom{f}\in \filter$,  $d(f)<\mu$,  and $(\forall \Sigma \subseteq X)(\Sigma \in \filter \iff \Sigma f\in \filter)$.

If $c(g)\geq \mu$,  then $c(fg)\geq \mu$ by Lemma \ref{lem:defect_properties} part \eqref{lem:defect_properties/defect<collapse}.  If $\dom{g} \not \in \filter$,  then $\dom{fg}=\dom{g}f^{-1} \not \in \filter$
since $f$ maps filter elements to filter elements and $(\dom{g}f^{-1})f\subseteq \dom{g}\not \in \filter$.  Finally,  if $g$ also satisfies $d(g)<\mu$ and $(\forall \Sigma \subseteq X)( \Sigma \in \filter \iff \Sigma g \in \filter)$,  then $d(fg)\leq d(f)+d(g)<\mu$ by Lemma \ref{lem:defect_properties} part \eqref{lem:defect_properties/defect_composition} and for any $\Sigma \subseteq X$ we have $\Sigma \in \filter \iff \Sigma f \in \filter \iff \Sigma fg \in \filter$.  Hence $fg\in  V_{\filter,\mu}$ in all cases. 
We have shown that $V_{\filter,\mu}$, and hence $\inv{V}_{\filter,\mu}$, are subsemigroups of $I_X$,  as required.

\vspace{\baselineskip}\noindent
{\bf (iii)} 
Fix a cardinal $\mu$ with $\min(\filter)<\mu\leq |X|^+$. We need to show that $V_{\filter, \mu}$ is not contained in any $S_{\nu}, \inv{S}_{\nu}, V_{\filter, \nu}$, or $\inv{V}_{\filter, \nu}$ over the relevant cardinals $\nu$. Since $\filter$ is an ultrafilter, there exists $\Sigma \in \filter$ which is a moiety of $X$. 
Let $f\in I_X$ be surjective (i.e. $d(f)=0$) with $\dom{f}= X \setminus \Sigma$. Then $\dom{f} \not \in \filter$ and $c(f)=|X|$. Hence $f \in V_{\filter, \mu}$ and $f \not \in \inv{S}_{\nu}$ for  $\nu=1$ and every infinite $\nu$ with $\nu \leq |X|$. Moreover, $f \not \in \inv{V}_{\filter, \nu}$  for every $\nu$ with $\min(\filter)<\nu\leq |X|^+$.

We now define $g_{\nu} \in I_X$ for every cardinal $\nu$ which satisfies $\mu \neq \nu \leq |X|^+$. If $\nu <\mu$,  then let $g_{\nu}$ fix $\Sigma$ pointwise, be total (i.e. $c(g_{\nu})=0$) and satisfy $d(g_{\nu})=\nu$.  If $\mu<\nu \leq |X|^+$, then let $g_{\nu} \in I_X$ satisfy $c(g_{\nu})=\mu$, $\Sigma \subseteq \dom{g_{\nu}}$, and $\im{g_{\nu}}=X \setminus \Sigma$.
Then $g_{\nu} \in V_{\filter, \mu} \cap \inv{V}_{\filter, \mu}$ for every $\nu$ with $\mu\neq \nu \leq |X|^+$. Moreover, $g_{\nu} \not \in S_{\nu}$ for $\nu=1$ and any infinite $\nu$ with $\mu \neq \nu\leq |X|$ and $g_{\nu} \not \in V_{\filter, \nu}$ for any infinite $\nu$ with $\mu\neq \nu \leq |X|^+$. 
It only remains to find $h \in V_{\filter, \mu} \setminus S_{\mu}$ in the case that $\mu\leq |X|$. Let $\Gamma \in \filter$ have cardinality $\min(\filter)<\mu$. By intersecting $\Gamma$ with $\Sigma$, we may assume that $|X \setminus \Gamma|=|X|$.  Let $h\in I_X$ have domain $X \setminus \Gamma$ and satisfy $d(h)=|X|$. Then $h \in V_{\filter, \mu} \cap \inv{V}_{\filter, \mu}$ but $h \not \in S_{\mu}$, as required. 

\vspace{\baselineskip}\noindent
{\bf (iv)} To see that $\inv{V}_{\filter, \mu}$ is indeed the inverse of $V_{\filter, \mu}$, first recall that $c(f^{-1})=d(f)$ and $\dom{f^{-1}}=\im{f}$ for any $f \in I_X$ and $\inv{\fin}=\fin$. So it suffices to show that the set $S=\set{f\in I_X}{(\forall \Sigma \subseteq X)(\Sigma \in \filter \iff \Sigma f\in \filter)}$ is closed under taking inverses. If $f\in S$, then $\im{f}\in \filter$, since $X \in \filter$. So if $\Sigma \subseteq X$, then $(\Sigma f^{-1})f =\Sigma \cap \im{f} \in \filter \iff \Sigma \in \filter$. Hence $\Sigma f^{-1} \in \filter \iff \Sigma \in \filter$, since $f\in S$. Hence $f^{-1}\in S$, as required. 

\vspace{\baselineskip}\noindent
{\bf (v)} Let $U$ be a subsemigroup of $I_X$ which contains $\Sym{X}_{\{\mathcal{F}\}}$ and is not contained in $V_{\filter, \mu}$, $\inv{V}_{\filter, \mu}$, $S_{\nu}$, or $\inv{S}_{\nu}$ for any $\mu$ and $\nu$ with $\min(\filter)<\mu \leq |X|^+$ and $\nu=1$ or $\aleph_0 \leq \nu \leq |X|$. By Lemma \ref{ultra_inj_ix}, there exists a total $f_0\in U$ such that $\im{f_0} \not \in \filter$. Note that by \eqref{ultrafilter_stabiliser_as_union}, $U$ contains every $a \in \Sym{X}$ with $\supp{a} \subseteq \im{f_0}$. Thus we may apply Lemma \ref{lem:constructing_f} to $f_0, \kappa = |X|$, and $U$ to find a total $f\in U$ with $d(f)=|X|$ and $\im{f}\subseteq \im{f_0} \not \in \filter$. Replacing $f$ by $f^2$ if necessary, we may ensure that $\im{f}$ is a moiety of $Y \cup \im{f}$. Since $Y \cup \im{f} \not \in \filter$, there exists $b \in \Sym{X}_{\{\mathcal{F}\}} \subseteq U$ such that $\im{fb} \subseteq Y$, as required.

\vspace{\baselineskip} We have shown that all five conditions hold. It follows from Theorem \ref{lem:total->all} that the maximal subsemigroups of $I_X$ which contain $\setStab{\Sym{X}}{\filter}$ but not $\Sym{X}$ are the semigroups $V_{\filter, \mu}$ and $\inv{V}_{\filter, \mu}$ over all infinite $\mu$ with $\mu\leq |X|^+$. Moreover, the maximal inverse semigroups of $I_X$ containing $\setStab{\Sym{X}}{\filter}$ are the maximal elements of the set of intersections $V_{\filter, \mu} \cap \inv{V}_{\filter, \mu}$ and $S_{\lambda} \cap \inv{S}_{\lambda}$ over the same values of $\mu$ and $\lambda=1$ or $\aleph_0 \leq \lambda \leq |X|$. 
It only remains to fix a cardinal $\mu$ with $\min(\filter)<\mu\leq |X|^+$ and show that $V_{\filter, \mu} \cap \inv{V}_{\filter, \mu}$ is maximal in this set of intersections. In the proof of condition (iii) above we already found $g_\nu \in V_{\filter, \mu} \cap \inv{V}_{\filter, \mu}$ for every $\nu$ with $\mu \neq \nu \leq |X|^+$ such that $g_\nu \not \in S_{\nu} \supseteq S_{\nu} \cap \inv{S}_{\nu}$ for $\mu \neq \nu \leq |X|$ and $g_{\nu} \not \in V_{\filter, \nu} \supseteq V_{\filter, \nu} \cap \inv{V}_{\filter, \nu}$ for any $\nu\neq \mu$. Moreover, in the case when $\mu \leq |X|$, we found $h \in V_{\filter, \mu} \cap \inv{V}_{\filter, \mu}$ such that $h \not \in S_{\mu} \supseteq S_{\mu} \cap \inv{S}_{\mu}$.
\end{proof}

\begin{corollary}\label{cor:count} The symmetric inverse monoid $I_X$ on an infinite set $X$ has $2^{2^{|X|}}$ maximal subsemigroups and $2^{2^{|X|}}$ maximal inverse subsemigroups.
\end{corollary}
\begin{proof}
Pospi\u sil's Theorem \cite[Theorem 7.6]{Jech2003} states that in \zfc, the cardinality of the set of uniform ultrafilters on any infinite set $X$ is $2^{2^{\vert\! X \!\vert}}$ (the same as the carnality of the powerset of $I_X$). By Theorem \ref{thm:maximal_ultrafilter}, each uniform ultrafilter corresponds to two maximal subsemigroups and one maximal inverse subsemigroup of $I_X$. By \eqref{ultrafilter_stabiliser_as_union},  distinct ultrafilters give rise to distinct maximal subgroups of $\Sym{X}$ and hence distinct maximal (inverse) subsemigroups of $I_X$. 
\end{proof}

\section{Stabilisers of finite partitions into moieties}

 Throughout this section, let $\partition = \{ \Sigma_0, \dots, \Sigma_{n-1} \}$ be a finite partition of an infinite set $X$ into $n\geq 2$ moieties. Recall that the stabiliser $\Stab{\partition}$ and the almost stabiliser $\AStab{\partition}$ of $\partition$ are the groups
    \begin{equation*}
        \Stab{\partition} = \set{f \in \Sym{X} }{ (\forall i \in n) (\exists j \in n) (\Sigma_i f = \Sigma_j)}
    \end{equation*}
    and
        \begin{align*}
        \AStab{\partition} = \set{f \in \Sym{X}}{\;&(\forall i \in n) (\exists j \in n)\\
        &(|\Sigma_i f \setminus \Sigma_j| +|\Sigma_j \setminus \Sigma_i f| < |X|)}.
    \end{align*}
It was shown in a note added in proof to \cite{subgroups_macpherson_neumann} that $\AStab{\partition} \supsetneq \Stab{\partition}$ is a maximal subgroup of $\Sym{X}$.
 We will classify the maximal subsemigroups and maximal inverse subsemigroups of $I_X$ which contain $\Stab{\partition}$. It will turn out that each such maximal (inverse) subsemigroup contains $\AStab{\partition}$. 

For $f \in I_X$, define the binary relation $\rho_f$ on $n=\{0,\dots, n-1\}$ by
    \begin{equation*}
    \rho_f = \set{(i,j) \in n \times n }{ \card{\Sigma_i f \cap \Sigma_j} = \card{X}}.
    \end{equation*}
Loosely speaking, the binary relation $\rho_f$ captures the action of $f$ on the parts of $\partition$. 
Recall that the \emph{binary relation monoid} $B_n$ consists of all binary relations on $n$ under the usual composition of binary relations defined by
$$\rho\sigma=\set{(i,j) \in n\times n}{(i,k) \in \rho \text{ and }(k,j) \in \sigma \text{ for some }k\in n}$$
for all $\rho, \sigma \in B_n$. The domain and image of a binary relation $\rho \in B_n$ are defined as
\begin{align*}
\dom{\rho}&=\set{i \in n}{(i,j) \in \rho \text{ for some }j \in n} \\ 
\im{\rho}&=\set{j \in n}{(i,j) \in \rho \text{ for some }i \in n}.
\end{align*}

We can now state the main result of this section.

\begin{theorem} \label{thm:maximal_astab}
Let $X$ be an infinite set and $\partition = \{\Sigma_0, \dots, \Sigma_{n-1}\}$ a finite partition of $X$ into $n\geq 2$ moieties.
Then the maximal subsemigroups of $I_X$ which contain $\Stab{\partition}$ but not $\Sym{X}$ are
\begin{align*}
    A_{\partition} &= \set{f \in I_X }{ \rho_f \in \Sym{n} \text{ or } \dom{\rho_f} \neq n} \\
    \inv{A}_{\partition} &= \set{f \in I_X }{ \rho_f \in \Sym{n} \text{ or } \im{\rho_f} \neq n}
\end{align*}
and the unique maximal inverse subsemigroups of $I_X$ containing $\Stab{\partition}$ but not $\Sym{X}$ is
\begin{equation*}
A_{\partition} \cap \inv{A}_{\partition} = \set{f \in I_X }{ \rho_f \in \Sym{n} \text{ or } \dom{\rho_f} \neq n \neq \im{\rho_f}}.
\end{equation*}
\end{theorem} The map $f \mapsto \rho_f$ is not a homomorphism from $I_X$ to $B_n$, but it is close to being one, in the following sense. 

\begin{lemma} \label{lem:compose_relations}
    Let $\partition$ be a finite partition into moieties of an infinite set $X$ and $f,g \in I_X$. Then $\rho_{fg}\subseteq \rho_f\rho_g$ and there exists $a \in \Stab{\partition}$ such that $\rho_{fag} = \rho_f \rho_g$.
\end{lemma}

\begin{proof}
Let $(i,j)\in \rho_{fg}$. Then $\Sigma_i fg \cap \Sigma_j=|X|$. Since $\partition$ is finite, there exists at least one $k\in n$ such that $|\Sigma_i f\cap \Sigma_k|=|\Sigma_k g \cap \Sigma_j|=|X|$. Hence $(i,j) \in \rho_f \rho_g$. 

To show that there exists $a \in \Stab{\partition}$ such that $\rho_{fag} = \rho_f \rho_g$, let $i \in n$ be arbitrary.
    If $j \in (i)\inv{\rho}_f$, then $\card{\Sigma_j f \cap \Sigma_i} = \card{X}$, and so $\Sigma_j f \cap \Sigma_i$ can be partitioned into $\card{(i)\rho_g} + 1$ moieties.
    If $k \in (i)\rho_g$, then $\card{\Sigma_k \inv{g} \cap \Sigma_i} = \card{X}$.
    Hence $\Sigma_k \inv{g} \cap \Sigma_i$ can be partitioned into $\card{(i)\inv{\rho}_f} + 1$ moieties.
    Let $a_i \in \Stab{\partition}$ be any element mapping one of the moieties partitioning $\Sigma_j f \cap \Sigma_i$ to one of the moieties partitioning $\Sigma_k \inv{g} \cap \Sigma_i$ for all $j \in (i)\inv{\rho}_f$ and for all $k \in (i)\rho_g$, while fixing everything else.
    The required $a \in \Stab{\partition}$ is then just the composite $a_0 \dots a_{n-1}$.
\end{proof}

We require the following two lemmas from \cite{maximal_east_mitchell_peresse} in the proof of Theorem \ref{thm:maximal_astab}.

\begin{lemma}[\cite{maximal_east_mitchell_peresse}, Lemma 9.3] \label{lem:nxn}
    Let $n$ be a natural number and $\rho, \sigma \subseteq n \times n$ binary relations such that $\dom{\rho}=\im{\sigma}=n$ but $\rho, \sigma \notin \Sym{n}$.
    Then the semigroup generated by $\Sym{n} \cup \{\rho, \sigma\}$ contains the relation $n \times n$.
\end{lemma}

\begin{lemma}[{\cite[Lemma 9.4]{maximal_east_mitchell_peresse}}] \label{lem:missing_all_part}
    Let $\partition$ be a finite partition of an infinite set $X$ into $n$ moieties and $f \in I_X$ be total and non-surjective (i.e. $c(f) = 0 < d(f)$).
    Then there exists a total $f^*$ in the semigroup generated by $\Stab{\partition} \cup \{f\}$ such that $\card{\Sigma_i \setminus \im{f^*}} \geq d(f)$ for all $i \in n$. If $d(f)$ is infinite, then $\card{\Sigma_i \setminus \im{f^*}} = d(f)$ for all $i \in n$.
\end{lemma}

\begin{proof}[Proof of Theorem \ref{thm:maximal_astab}]

If $Y=\Sigma_0$, then $\Stab{\partition}$ satisfies equation \eqref{lem:total->all/YSym} of Theorem \ref{lem:total->all}. So it suffices to show that $A_{\partition}$ and $\inv{A}_{\partition}$, together with the semigroups described in Theorem \ref{thm:maximal_sym}, satisfy the conditions of Theorem \ref{lem:total->all} where $G=\Stab{\partition}$.

\vspace{\baselineskip}\noindent
{\bf (i)} Note that if $f\in \Sym{X}$, then $\dom{\rho_f}=\im{\rho_f}=n$ and $\rho_f \in \Sym{n} \iff f \in \AStab{\partition}$. It follows that 
\begin{equation}\label{contains_astab}
A_{\partition} \cap \Sym{X}=\inv{A}_{\partition} \cap \Sym{X}=\AStab{\partition}\supseteq \Stab{\partition}.
\end{equation}

\vspace{\baselineskip}\noindent
{\bf (ii)} By \eqref{contains_astab} above, $A_{\partition}$ and $\inv{A}_{\partition}$ are proper subsets of $I_X$. Let $f,g \in A_{\partition}$. If $\dom{\rho_f} \subsetneq n$, then $\dom{\rho_f \rho_g} \subseteq \dom{\rho_f}\subsetneq n$. Similarly, if $\rho_f \in \Sym{n}$ and $\dom{\rho_g} \subsetneq n$, then $\dom{\rho_f\rho_g}=\dom{\rho_g}\inv{\rho}_f \subsetneq n$ since $n$ is finite. In both cases it follows from Lemma \ref{lem:compose_relations} that 
$$\dom{\rho_{fg}} \subseteq \dom{\rho_f \rho_g} \subsetneq n$$
and so $fg \in A_{\partition}$. So assume that $\rho_f, \rho_g \in \Sym{n}$. Then $\rho_{fg}$ is a subset of the permutation $\rho_f\rho_g$. Hence either $\rho_{fg}=\rho_f\rho_g \in \Sym{n}$ or $\dom{\rho_{fg}} \subsetneq n$. Thus $fg \in A_{\partition}$.

\vspace{\baselineskip}\noindent
{\bf (iii)} If $f\in I_X$ satisfies $\dom{f}=\Sigma_0$ and $\im{f}=X$, then $f \in A_{\partition} \setminus \inv{A}_{\partition}$. If $g\in I_X$ is total and maps each $\Sigma_i$ to a moiety of $\Sigma_i$, then $\rho_g=\rho_{\inv{g}}$ is the identity on $n$ and so $g, g^{-1} \in A_{\partition}$. In other words, $g\in A_{\partition} \cap \inv{A}_{\partition}$. However, $g$ is not contained in any $S_{\mu}$ and $g^{-1}$ is not contained in any $\inv{S}_{\mu}$ defined in Theorem \ref{thm:maximal_sym}.

\vspace{\baselineskip}\noindent
{\bf (iv)} Note that $\rho_{f^{-1}}=\set{(j,i)}{(i,j) \in \rho_f}$ for any $f \in I_X$. That is, 
 $\rho_{f^{-1}}$ is the so-called converse of the binary relation $\rho_f$. In particular, $\rho_f \in \Sym{n} \iff \rho_{f^{-1}} \in \Sym{n}$ and $\dom{\rho_{f^{-1}}}=\im{\rho_f}$. Thus $\inv{A}_{\partition}$ is indeed the inverse of $A_{\partition}$.

\vspace{\baselineskip}\noindent
{\bf (v)} Let $U$ be a subsemigroup of $I_X$ such that $\Stab{\partition} \subseteq U$ but $U$ is not contained in $A_{\partition}$, $\inv{A}_{\partition}$, $S_1$, or $S_{\mu}$ for any infinite $\mu \leq |X|$.
The first step is to show, by transfinite induction, that for every cardinal $\mu\leq |X|$
\begin{equation}\label{finite_partition_induction}
\text{there exists $f_{\mu}\in U$ with $c(f_{\mu})=0$ and $|\Sigma_i \setminus \im{f_{\mu}}|\geq \mu$ for all $i\in n$.}
\end{equation}
Let $f\in U \setminus S_1$. Then $f$ is total and $d(f)>0$. If $\mu$ is finite, then $c(f^{\mu})=0$ and $d(f^{\mu})\geq \mu$ by Lemma  \ref{lem:defect_properties} parts \eqref{lem:defect_properties/collapse_composition} and \eqref{lem:defect_properties/defect_total}. Hence, by Lemma \ref{lem:missing_all_part}, condition \eqref{finite_partition_induction} holds for all finite $\mu$. So let $\nu\leq |X|$ be infinite and assume that \eqref{finite_partition_induction} holds for all $\mu<\nu$. Let $g_{\nu}\in U \setminus S_{\nu}$, let $\mu=c(g_{\nu})<\nu$, and let $f_{\mu}$ satisfy \eqref{finite_partition_induction}. Since $c(g_{\nu})=\mu$ and $|\Sigma_i \setminus \im{f_{\mu}}|\geq \mu$ for all $i\in n$, there exists $a\in \Stab{\partition}$ which takes $\im{f_{\mu}}$ into the domain  of $g_{\nu}$. Then $f_{\mu}ag_{\nu}$ is total and $d(f_{\mu}ag_{\nu})\geq d(g_{\nu})\geq \nu$. Applying Lemma \ref{lem:missing_all_part} to $f_{\mu}ag_{\nu}$, we conclude that \eqref{finite_partition_induction} holds for $\nu$. Hence  \eqref{finite_partition_induction} holds for all $\mu\leq |X|$. In other words, there exists a total $f_{|X|}\in U$ such that  $|\Sigma_i \setminus \im{f_{|X|}}|= |X|$ for all $i\in n$. 

Let $g \in U \setminus A_{\partition}$ and $h \in U \setminus \inv{A}_{\partition}$. Then $\dom{\rho_g}=\im{\rho_h}=n$ but $\rho_g$ and $\rho_n$ are not permutations. By Lemma \ref{lem:nxn}, the semigroup generated by $\Sym{n} \cup \{\rho_g, \rho_h\}$ contains $n\times n$. Since $\Stab{\partition}$ induces $\Sym{n}$ via the map $f \mapsto \rho_f$, Lemma \ref{lem:compose_relations} implies that there exists $t$ in the semigroup generated by $\{g, h\}\cup \Stab{\partition}\subseteq U$ such that $\rho_t=n\times n$. In particular, $|\Sigma_i t \cap \Sigma_0|=|X|$ for every $i \in n$. Let $a\in \Stab{\partition}$ map $\im{f_{|X|}}$ into $\Sigma_0 t^{-1} \cap \Sigma_i$ for every $i\in n$. Then $f_{|X|}at$ is the required total element of $U$ with image in $Y=\Sigma_0$.

\vspace{\baselineskip} We have shown that the five conditions are satisfied and so $A_{\partition}$ and $\inv{A}_{\partition}$ are the only maximal  subsemigroups of $I_X$ containing $\Stab{\partition}$ but not $\Sym{X}$. Moreover, in the proof of part (iii) above, we found $g\in A_{\partition} \cap \inv{A}_{\partition}$ such that $g \not \in S_{\mu}$  for $\mu=1$ or any infinite $\mu\leq |X|$. In particular $A_{\partition} \cap \inv{A}_{\partition}$ is not contained in any $S_{\mu}\cap \inv{S}_{\mu}$. Thus, $A_{\partition} \cap \inv{A}_{\partition}$ is the unique maximal inverse subsemigroup of $I_X$ containing $\Stab{\partition}$ but not $\Sym{X}$.
\end{proof}

\section{Acknowledgements}

The authors would like to thank the anonymous reviewers for their careful reading of the paper and insightful comments which have improved the presentation of this paper.

    \hfill
	\bibliographystyle{plain}
	\bibliography{Kilder}

\vspace{1.5em}

\noindent
\authorname{Martin Hampenberg}. \affiliation{Department of Physics, Astronomy and Mathematics, University of Hertfordshire, Hatfield AL10 9AB, United Kingdom}\\
\email{m.christensen2@herts.ac.uk}

\vspace{1.5em}

\noindent
\authorname{Yann Péresse}. \affiliation{Department of Physics, Astronomy and Mathematics, University of Hertfordshire, Hatfield AL10 9AB, United Kingdom}\\
\email{y.peresse@herts.ac.uk}

\end{document}